\documentclass[12pt]{amsart}
\usepackage{amscd,amssymb,verbatim}
\theoremstyle{plain}

\usepackage{color}
\numberwithin{equation}{section}

\newcommand{\N}{\mathbb{N}}
\newcommand{\R}{\mathbb{R}}

\newcommand{\e}{\varepsilon}
\newcommand{\norm}[1]{\left\lVert#1\right\rVert}
\newcommand{\U}{\mathfrak{U}}

\newcommand{\s}{\sigma}
\newcommand{\la}{\langle}
\newcommand{\ra}{\rangle}
\newcommand{\Tree}{\mathcal{T}}
\newcommand{\B}{\mathcal{B}}
\newcommand{\Lc}{\mathcal{L}}
\DeclareMathOperator{\supp}{supp}
\DeclareMathOperator{\ran}{ran}

\newtheorem{theorem}{Theorem}[section]
\newtheorem{lemma}[theorem]{Lemma}
\newtheorem{remark}[theorem]{Remark}

\newtheorem{definition}[theorem]{Definition}
\newtheorem{proposition}[theorem]{Proposition}
\newtheorem{corollary}[theorem]{Corollary}
\newtheorem{notation}[theorem]{Notation}
\newtheorem*{claim}{\textit{Claim}}
\newtheorem*{question}{Question}

\begin{document}

\title[Copies of $\ell_2$ and $J$ in $JT$]{Embedding $\ell_2$ and $J$ into subspaces of $JT$ and $JT^*$}

\author{Spiros A. Argyros}
\address{National Technical University of Athens, Faculty of Applied Sciences, Department of Mathematics, Zografou Campus, 157 80, Athens, Greece.}
\email{sargyros@math.ntua.gr}

\author{Manuel Gonz\'alez}
\address{Department of Mathematics, Faculty of Sciences, University of Cantabria, Avenida de los Castros s/n, 39071 Santander, Spain.}
\email{manuel.gonzalez@unican.es}

\author{Pavlos Motakis}
\address{Department of Mathematics and Statistics, York University, 4700 Keele Street, Toronto, Ontario, M3J 1P3, Canada}
\email{pmotakis@yorku.ca}

\date{\today}

\begin{abstract}
In the first part of the paper we show that every closed subspace of $JT$ or $JT^*$ contains $\ell_2$ complemented in $JT$ or $JT^*$ respectively, and $JT$ contains uncomplemented copies of $\ell_2$. As a result, the predual $\B$ of $JT$, as well as the spaces $JT$ and $JT^*$, are subprojective and superprojective. In the second part, we prove that every weakly Cauchy sequence that is not weakly convergent in $JT$ has a subsequence equivalent to the basis of $J$. Hence, every non-reflexive subspace of $JT$ contains an isomorphic copy of $J$, and every Schauder basic sequence in $JT$ has a subsequence which is equivalent either to the basis of $\ell_2$ or to the basis of $J$. Moreover these subspaces may be selected  to be complemented in $JT$.
\end{abstract}

\maketitle

\section{Introduction}\label{sect:intro}

The James tree space $JT$, introduced more than half a century ago by R.C. James \cite{James:74}, has become a central example for understanding the structure of Banach spaces beyond the classical ones. Its fundamental properties include  that $JT$  is separable, it contains no copies of $\ell_1$ and $JT^*$ is non-separable, as proved by James. Furthermore, $JT$ is the dual of a Banach space $\B$, and the quotient $JT^*/ \B$ is isometric to $\ell_2 (2^{\N})$ (due to J. Lindenstrauss and C. Stegall \cite{LS}), which yields the decomposition $JT^{**} = JT \oplus \ell_2 (2^{\N})$. Finally, every seminormalized weakly null sequence in $JT$ has a subsequence equivalent to $\ell_2$, a result established by I. Amemiya and I. Ito \cite{AI}.

The norm of $JT$ is defined on the space of real-valued functions on the dyadic tree $\Tree$ with finite support as follows:
\[ \norm{x}_{JT} = \sup_{\mathcal{P}} \left( \sum_{S \in \mathcal{P}} \left( \sum_{\alpha \in S} x(\alpha) \right)^2 \right)^{1/2} \]
where the supremum is taken over all families $\mathcal{P}$ of pairwise disjoint segments in $\Tree$.

The space $JT$ is defined as the completion of the space $(c_{00} (\Tree), \| \cdot\|_{JT} )$. Notably, for every branch $\s$ of $\Tree$, the sequence $(e_\alpha)_{\alpha \in \s}$ is isometrically equivalent to the basis of the James quasi-reflexive space $J$ \cite{James:51}. Consequently, $(e_\alpha)_{\alpha\in \s} \xrightarrow{w^*} e_\s^{**}$, and the family $(e_\s^{**})_{\s \in 2^\N}$ is isometrically equivalent to the standard basis of $\ell_2(2^\N)$. Furthermore, every sequence $(e_{\alpha_n})_n$ with $(\alpha_n)_n$ a sequence  of pairwise incomparable nodes is isometrically equivalent to the $\ell_2$ basis.

These properties are not merely specific to $JT$; rather, they reflect the underlying structure of the broader class of separable Banach spaces that do not contain $\ell_1$ and possess a non-separable dual, as indicated by the following theorem:

\begin{theorem}[Argyros, Dodos \& Kanellopoulos \cite{ADK}]
Let $X$ be a separable Banach space containing no copies of $\ell_1$ with $X^*$ non-separable. Then there exists a tree-basis $(e_\alpha)_{\alpha\in \Tree}$ in $X$ satisfying the following properties:
\begin{enumerate}
    \item For every branch $\s$ of $\Tree$, we have $(e_{\alpha})_{\alpha \in \s} \xrightarrow{w^*} x_\s^{**} \ne 0$.
    \item For every sequence $(\alpha_n)_n$ of pairwise incomparable nodes, we have $e_{\alpha_n} \xrightarrow{w} 0$.
    \item The family $(x_\s^{**})_{\s \in 2^\N}$ is $w^*$-accumulated to $0$ and it is unconditional.
\end{enumerate}
\end{theorem}

At the boundaries of this class, one encounters exotic separable spaces, such as W.T. Gowers' tree space $GT$  \cite{G}, such that every infinite-dimensional closed subspace contains a tree-basis satisfying the aforementioned properties. This actually yields that $GT$ does not contain $\ell_1$ and has non-separable dual. For a comprehensive study of such spaces, we refer the reader to \cite{AAT}.

In the present paper, we complete the study of the isomorphic structure of the subspaces of $JT$, providing a further understanding of the isomorphic structure of its natural relatives $\mathcal{B}$, $JT^*$ and $JT^{**}$. As a consequence, we establish the subprojectivity and superprojectivity of $\mathcal{B}$, $JT$, and $JT^*$. In particular, we prove that every infinite-dimensional closed subspace of $JT$ or $JT^*$ contains an isomorphic copy of $\ell_2$ that is complemented in the respective ambient space. Furthermore, every non-reflexive subspace of $JT$ contains an isomorphic copy of the James space $J$ which is also complemented in $JT$. We also recall that, as shown in \cite{AAK:08}, every subspace of $JT$ with a non-separable dual contains a subspace isomorphic to $JT$ and complemented in $JT$.

The following theorem encapsulates these results and provides a complete description of the isomorphic structure of the subspaces of $JT$:

\begin{theorem}
The following facts hold:
\begin{enumerate}
    \item Every infinite-dimensional closed subspace of $JT$ or $JT^*$ contains a further subspace isomorphic to $\ell_2$ that is complemented in $JT$ or $JT^*$, respectively.
    \item Every closed non-reflexive subspace of $JT$ contains an isomorphic copy of $J$ which is complemented in $JT$.
    \item Every closed subspace of $JT$ with a non-separable dual contains a subspace $Y$ isomorphic to $JT$ that is complemented in $JT$.
\end{enumerate}
\end{theorem}

The embeddability of $\ell_2$ into the subspaces of $JT$ was first established by James \cite{James:74}, and Amemiya and Ito later extended this by showing that every seminormalized weakly null sequence in $JT$ admits a subsequence equivalent to the basis of $\ell_2$. The following result serves as a further extension of the Amemiya-Ito theorem:

\begin{theorem}
Every seminormalized basic sequence in $JT$ has a subsequence equivalent to the standard basis of either $\ell_2$ or $J$, and the closed subspace it generates is complemented in $JT$.
\end{theorem}

We also prove the following projectivity properties (see Definition \ref{definition1}):

\begin{theorem}
The spaces $\mathcal{B}$, $JT$, $JT^*$, and all of their higher-order dual spaces are both subprojective and superprojective.
\end{theorem}

The paper is organized as follows. After the preliminaries in Section 2, Section 3 presents a revisited approach to the basic properties of $JT$. Our proofs are based on the Amemiya-Ito theorem and a result implicitly contained in \cite{LS}:

\begin{proposition}
Every seminormalized block sequence $(x_n)_n$ in $JT$ has a subsequence $(x_{n_k})_k$ such that $\lim_k \s^*(x_{n_k}) = 0$, with the exception of countably many branches of $\Tree$.
\end{proposition}

This result, coupled with the Amemiya-Ito theorem, easily yields that $\ell_1$ is not contained in $JT$, and the remaining basic properties of $JT$ emerge as straightforward consequences. Section 3 also includes the following proposition, which is critical for proving that $JT^*$ is superprojective:

\begin{proposition}
Let $JT^{**} = JT \oplus \ell_2(2^\N)$ be the natural decomposition. Then, for every separable subspace $Y$ of $\ell_2(2^\N)$, $Y_\perp$ is a complemented subspace of $JT^*$.
\end{proposition}

In the first part of Section 4, we prove that the spaces $JT$ and $JT^*$ are $\ell_2$-complement\-ably saturated, which yields that the spaces $\mathcal{B}$, $JT$, and $JT^*$ are subprojective. In the second part, we prove the superprojectivity of these spaces.

In the final section, we first prove that non-trivial weakly Cauchy sequences in the James space $J$ possess a subsequence equivalent to the basis of $J$, and the generated subspace is complemented in $J$. The second part of the section is devoted to proving the analogous result for sequences in $JT$.

\section{Preliminaries}

We use standard Banach space notation. For a Banach space $X$, $B_X$ and $S_X$ denote the closed unit ball and the unit sphere of $X$. The closed subspace generated by a sequence $(x_n)\subset X$ is denoted by $[x_n]$. A sequence $(x_n)$ in $X$ is seminormalized if $0<\inf_n \|x_n\|$ and $\sup_n\|x_n\|<\infty$.

Given non-empty subsets $A\subset X$ and $G\subset X^*$, we denote
$A^\perp= \{f\in X^* : f(a)=0\; \textrm{ for all } a\in A\}$ and $G_\perp= \{x\in X : g(x)=0\; \textrm{ for all } g\in G\}$.

If $X$ and $Y$ are Banach spaces, we write $X\simeq Y$ if they are isomorphic, and $\Lc(X,Y)$ denotes the set of all (bounded) operators from $X$ into $Y$.
Given a closed subspace $M$ of a Banach space $X$, $J_M:M\to X$ and $Q_M:X\to X/M$ denote the inclusion and the quotient map.

We say that a sequence $(x_n)$ in $X$ \emph{admits an upper $\ell_2$-estimate} if there exists a constant $C>0$ such that $\|x_1+\cdots+ x_k\| \leq C(\|x_1\|^2+\cdots+\|x_k\|^2)^{1/2}$ for every $k\in\N$.
The sequence $(x_n)$ \emph{admits a lower $\ell_2$-estimate} if there exists a constant $D>0$ such that $\|x_1+\cdots+ x_k\| \geq D (\|x_1\|^2+\cdots+ \|x_k\|^2)^{1/2}$ for every $k\in\N$.

Given a sequence $(X_n)$ of Banach spaces, we denote
$$
\ell_2(X_n) =\{(x_n) : x_n\in X_n \textrm{ and } (\|x_n\|)\in\ell_2\}.
$$
Endowed with the norm $\|(x_n)\|_2 = (\sum_{n=1}^\infty \|x_n\|^2)^{1/2}$, $\ell_2(X_n)$ is a Banach space.

\subsection{The dyadic tree}
Recall that the dyadic tree $T$ is a countable tree with a unique root, and every $\alpha \in T$ has two immediate successors. Traditionally, $T$ is represented as $T=\{(n,i) : 0\leq n<\infty, 0\leq i <2^n \}$ with the usual partial order in which each $(n,i)\in T$ has two successors $(n+1,2i)$ and $(n+1,2i+1)$.
The \emph{height of} $\alpha=(n,i)\in T$ is $|\alpha|=n$.

Working in Banach spaces related to the dyadic tree, it is convenient to identify $T$ with $(\N, \preceq )$ where $\preceq$ is compatible with the natural order of $\N$ in the sense that $n \preceq m$ implies $n < m$. For example, the function $V : T \to \N$ defined by the rule $V((n,i)) = 2^n + i$ identifies $T$ with $(\N, \preceq)$.

A segment of $T$ is any set $s$ such that for every $n, m \in s$ we have $n \preceq m$, and for every $k\in \N$ with $n \preceq k \preceq m$ we have $k\in s$. We shall denote the segments by the letters $s, t$.
An initial segment is a segment having the root as its first element.
If $s$ is a segment and $I$ is an interval of $\N$, the restriction of $s$ to $I$ is the segment $s_{|I}=s\cap I$. This follows from the compatibility of $\preceq$ and $\leq$.

A branch is a maximal segment. We shall denote branches by the letters $\sigma,\tau$.
If $\sigma$ is a branch and $n\in\N$, we denote by $\sigma_{> n}$ the set $\sigma_{> n}=\{k\in \sigma: k>n\}$. For a branch $\sigma$ and $n\in\N$, we shall call the set $\sigma_{> n}$ a final segment.

Also, by $\sigma_{< n}$ we shall denote the initial segment $\sigma_{<n}=\{k\in \sigma: k<n\}$. In a similar manner, we define the initial segment $\sigma_{\leq n}$ and the final segment $\sigma_{\geq n}$.

Note that every infinite segment is contained in a unique branch. Actually, it is a final segment of a branch $\sigma$.
Two final segments $\sigma_{1,>n_{1}}$ and $\sigma_{2,>n_{2}}$ are said to be incomparable if their first elements are incomparable.
We denote by $\Gamma$ the set of all branches in $T$. Note that the cardinality of $\Gamma$ is $\mathfrak{c}$.

Let us note that:
\begin{enumerate}
    \item[a)] if $\sigma_{1},\dots,\sigma_{l}$ are distinct branches, then there exists $k\in\N$ such that the final segments $\sigma_{1,> k},\dots, \sigma_{l,>k}$ are incomparable.
    \item[b)] if $(\s_{n})_{n\in\N}$ is a sequence of initial segments, there exist a subsequence $(\s_{n})_{n\in L}$ and an initial segment or branch $\sigma$ such that $\s_{n}\xrightarrow[n\in L]{} \s$ in the pointwise topology.
\end{enumerate}

\section{James Tree and its basic properties}\label{sect:prelim}

In this section, we define the norm of the space $JT$ and we present some of its basic properties. Namely, we prove that $\ell_1$ is not contained in $JT$ and that the quotient $JT^*/ \B$ is isometric to $\ell_2(\Gamma^*)$ and their consequences. Our approach is based on the Amemiya-Ito theorem \cite{AI}. In particular, we explain how, using the Amemiya-Ito Theorem (Theorem \ref{A-I}), we derive an alternative approach to the proof of the basic properties of $JT$ and $JT^*$ (Theorems \ref{TH_JT}, \ref{P2}). For additional information on these spaces, we refer to \cite[Chapter 3]{FetterG} or \cite[Chapter VIII]{Dulst}.

\subsection{The norm of $JT$}
Let $\Tree$ denote the dyadic tree. We start with the definition of the norm of $JT$.

For $x=(x_i)_i \in c_0$, we denote
\begin{equation}\label{JT-norm}
\|x\|_{JT}= \sup \Big(\sum_{j=1}^k \big(\sum_{i\in s_j} x_{i}\big)^2 \Big)^{1/2},
\end{equation}
where the \emph{sup} is taken over all $k\in\N$ and $s_1,\ldots, s_k$ pairwise disjoint segments in $T$.

The space $JT$ is defined as
$$
JT=\{x=(x_i)_i\in c_0: \|x\|_{JT}< \infty \},
$$
and endowed with the norm $\|\cdot\|_{JT}$ is a Banach space.

Let $F$ be a finite or infinite initial segment of $\N$. A family $P=\{s_i\} _{i\in F}$  of pairwise disjoint segments of  $\mathcal{T}$ is called a partition of $\mathcal{T}$ and for $x\in JT$   and $P$ a partition  we set
\begin{equation}\label{JT-P}
\|x\|_{P}= \Big(\sum_{i\in F} \big(\sum_{j\in s_i} x_{j}\big)^2 \Big)^{1/2} \le \|x\|_{JT}.
\end{equation}

We set $(e_n)_n$ the natural basis of $c_0$ which is a normalized bimonotone boundedly complete basis of $JT$ (see, e.g., \cite[page 419]{AI} or \cite[page 10]{andrew}).
We denote by $\{e^*_n : n \in \N\}\subset JT^*$ the sequence of biorthogonal functionals.
Every finite segment $s$ of $T$ defines a norm-one functional $s^* \in JT^*$ by the rule $s^*(x) = \sum_{n\in s}x(n)$. This is extended to the infinite segments and hence to the branches. We set $\Gamma^* = \{ \sigma^* : \sigma \in \Gamma \}$. It is easy to see that for $\sigma_1,\sigma_2 \in \Gamma$ not equal, $\|\sigma_1^* - \sigma_2^*\| > 1$; hence we have:

\begin{proposition}
The space $JT^*$ is non-separable.
\end{proposition}

The following result is well-known.

\begin{proposition}\label{P_A}
Let $A\subset \Tree$ such that for each segment $s$ in $\Tree$ the intersection $s\cap A$ is again a segment. Then $P_A(\sum_{k=1}^\infty a_k e_{k}) = \sum_{k=1; k\in A}^\infty a_k e_{k}$ defines a norm-one projection $P_A:JT\to JT$.
\end{proposition}
\begin{proof}
Observe that $\sum_{j=1}^k \big(\sum_{i\in s_j} (P_Ax)_{}\big)^2 = \sum_{j=1}^k \big(\sum_{t_i\in s_j\cap A} x_{i}\big)^2$.
\end{proof}

We will need the following fundamental result of Amemiya and Ito \emph{(\cite[Theorem]{AI})}.

\begin{theorem}\label{A-I}
If $(x_n)_n$ is a seminormalized block sequence in $JT$ such that $\sigma^*(x_n) \to 0$ for every $\sigma \in \Gamma$, then there exists a subsequence $(x_{n_k})_k$ of $(x_n)_n$ equivalent to the unit vector basis of $\ell_2$. In particular, $(x_n)_n$ is a weakly null sequence.
\end{theorem}

This result is a key ingredient for understanding the structure of $JT$. Its proof is based on the Ramsey Theorem for two-point sets and an iteration.

Next, we show that it also provides short and clear proofs of the basic properties of $JT$ and $JT^*$. See \cite{AMM} for variants of Theorem \ref{A-I}.

The following definition and proposition are implicitly included in \cite{LS}.

\begin{definition}\label{D2}
A bounded sequence $(x_n)_n$ in $JT$ admits a non-zero $\ell_2$-vector if there exists sequences $(\alpha_i)_{i \in F}$, $(\sigma_i)_{i\in F}$ with either $F=\N$ or $F=\{1,\dots k\}$ such that $\alpha_i \neq 0$, $\sigma_i^*(x_n) \to \alpha_i$ and for every $\sigma \notin \{\sigma_i:i\in F\}$ we have $\sigma^*(x_n) \to 0$.
\end{definition}

We denote by $\#G$ the cardinal of the set $G$. The following fact is easily proved.

\begin{lemma}\label{L1}
Assume that $(x_n)_n$ is a normalized block sequence in $JT$. For $\epsilon > 0$ we set $G_{\epsilon} = \{ \sigma : \liminf_n |\sigma^*(x_n)| \geq \epsilon \}$.
Then $\#G_{\epsilon}\leq 1/\epsilon^2$.
\end{lemma}

\begin{proposition}\label{P1}
Let $(x_n)_n$ be a seminormalized block sequence in $JT$. Then one of the following two alternatives hold:
\begin{enumerate}
    \item[a)] There exists a weakly null subsequence $(x_{n_k})_k$ of $(x_n)_n$.
    \item[b)] (i) There exists a subsequence $(x_{n_k})_k$ admitting a non-zero $\ell_2$-vector.\\
    (ii) If $(\alpha_i)_{i \in F}$, $(\sigma_i)_{i\in F}$ determine the $\ell_2$-vector, then $\liminf_k \|x_{n_k}\| \geq \|(\alpha_i)\|_2$.
\end{enumerate}
\end{proposition}
\begin{proof}
Assume that (a) does not hold. Then using Lemma \ref{L1} and iteration we choose $(L_j)_{j\in \N}$ a decreasing sequence of infinite subsets of $\N$ and $(G_j)_{j\in\N}$ an increasing sequence of finite subsets of $\Gamma$ such that the following hold:

(1) For every $\sigma \in G_j$, we have $\lim_{n\in L_j}\sigma^*(x_n) = \alpha_{\sigma}$ and $|\alpha_{\sigma}| \geq \frac{1}{j}$.

(2) For every $\sigma \notin G_j$, $\liminf_{n\in L_j} |\sigma(x_n)| < \frac{1}{j}$.

Choose an infinite $L$ almost contained in each $L_j$ and set $G = \cup_j G_j$. The set $G$ is countable and it is easy to see that for $\sigma \in G$, $\lim_{n\in L} \sigma^*(x_n) = \alpha_{\sigma} \neq 0$ while for $\sigma \notin G$, $\lim_{n\in L} \sigma^*(x_n) =0$.


We now prove the last part of the statement. Fix $N\in F$ and choose $m\in\mathbb{N}$ such that the segments $\sigma_{1,>m},\ldots,\sigma_{N,>m}$ are pairwise disjoint. Define the partition $P_{N} = \{\sigma_{i,>m}\}_{i=1}^N$. Because $(x_{n_k})_k$ is a block sequence, for $k$ sufficiently large, we have that $(\sigma_{i,>m}(x_{n_k}))_{i=1}^N = (\sigma_i(x_{n_k}))_{i=1}^N$.  Therefore,
\begin{align*}
\liminf_k\|x_{n_k}\| &\geq \liminf_k\|x_{n_k}\|_{P_N} = \liminf_k\Big(\sum_{i=1}^N|\sigma^*_{i,>m}(x_{n_k})|^2\Big)^{1/2}\\
& = \liminf_k\Big(\sum_{i=1}^N|\sigma^*_{i}(x_{n_k})|^2\Big)^{1/2} = \Big(\sum_{i=1}^Na_i^2\Big)^{1/2}. \qedhere
\end{align*}
\end{proof}

Proposition \ref{P1} and Theorem \ref{A-I} yield the following basic property of $JT$ proved by James in \cite{James:74}.

\begin{theorem}\label{TH_JT}
The space $JT$ does not contain an isomorphic copy of $\ell_1$.
\end{theorem}
\begin{proof}
As is well known and follows from a standard gliding hump argument, whenever $\ell_1$ is isomorphic to a subspace of a Banach space with a basis, there exists a block sequence of that basis equivalent to the unit vector basis of $\ell_1$. In particular, if $\ell_1$ is embedded into $JT$ then there exists a seminormalized block sequence $(x_n)_n$ equivalent to the $\ell_1$ basis. Then Proposition \ref{P1} implies that there exists a subsequence $(x_{n_k})_k$ of $(x_n)_n$ admitting a non-zero $\ell_2$-vector. Setting $y_k = x_{n_{2k+1}} - x_{n_{2k}}$, we have $\sigma^*(y_k) \to 0$ for every $\sigma \in \Gamma$. Theorem \ref{A-I} yields that $(y_k)_k$ is a weakly null sequence, a contradiction.
\end{proof}

\subsection{The spaces $\B$, $JT^*$}
We denote $\B=[e^*_{k}]$, the closed subspace of $JT^*$ generated by the sequence $(e^*_{k})$, and $\|\cdot\|_\B$ the norm on $\B$. Note that $\{e^*_{k} : k \in \N\}$ is a shrinking basis of $\B$. Therefore $\B^*= JT$ and, by \cite[Proposition 1.b.2]{L-T:77}, we have
\begin{equation}\label{normJT*}
JT^* = \Big\{x=(x_t)_{t\in T}\subset \R : \sup_{n\in\N} \Big\|\sum_{i=1}^n x_{t_i} e^*_{t_i}\Big\|_{\B}< \infty \Big\}.
\end{equation}

\begin{proposition}\label{P3}
The family $D = \{ e_n^* :n\in \N\} \cup \{ \sigma^*: \sigma \in \Gamma\}$ norm-generates $JT^*$.
\end{proposition}
\begin{proof}
Assume that the conclusion fails. Then there exists an $x^{**} \in JT^{**}$ such that $x^{**} \neq 0$ and $D \subset \ker(x^{**})$.
Since $\ell_1$ is not isomorphic to a subspace of $JT$, there exists a bounded sequence $(x_n)_n$ in $JT$ such that $x_n \xrightarrow{w^*} x^{**}$ (\cite{OR}). Since $e_k^*(x_n) \to 0$ we may assume that $\{x_n\}_n$ is a block sequence \cite[Proposition 1.3.10]{Al-K:16}. Then for every $\sigma \in \Gamma$, we have that $\sigma^*(x_n) \to 0$ and Theorem \ref{A-I} yields that $x_n \xrightarrow{w} 0$, a contradiction.
\end{proof}

The following is also well-known.

\begin{lemma}\label{L3}
The set
\begin{equation}
W = \Bigl\{ \sum_{i=1}^n \lambda_i s_i^* : \sum_{i=1}^n \lambda_i^2 \leq 1 \text{ and } \{s_i\}_{i=1}^n \text{ pairwise disjoint segments} \Bigr\}
\end{equation}
is a subset of $B_{JT^*}$ which is norming for $JT$.
\end{lemma}
\begin{proof}
Let $x^* = \sum_{i=1}^n \lambda_i s_i^* \in W$ and $x\in J$ with $\|x\|=1$. Then Cauchy-Schwarz inequality yields
\[ |x^*(x)| = \Big| \sum_{i=1}^n \lambda_i s_i^*(x) \Big| \leq \Big( \sum_{i=1}^n \lambda_i^2\Big)^{1/2} \Big( \sum_{i=1}^n s_i^*(x)^2\Big)^{1/2} \leq 1, \]
which shows that $W \subset B_{JT^*}$.

Also for $x\in c_{00}$ with $\|x\| = 1$, choose $\{s_i\}_{i=1}^n$ pairwise disjoint segments such that
\[ \Big( \sum_{i=1}^n s_i^*(x)^2\Big)^{1/2} = 1. \]
Then $x^* = \sum_{i=1}^n s_i^*(x)s_i^*\in W$ and $x^*(x) = 1$. This completes the proof.
\end{proof}

The next lemma is implicitly contained in \cite{LS} and will be used for the description of the norm of $JT^*/\B$.

\begin{lemma} \label{L2}
\begin{enumerate}
    \item[a)] Let $\{s_i^*\}_{i=1}^k$ be a disjoint family of segments of $\Tree$. Assume also that there exists a family $\{n_i\}_{i=1}^k$ of incomparable nodes of $\Tree$ such that $n_i \in s_i$, $i= 1 \dots k$. Then for $\{\lambda_i\}_{i=1}^k$ we have
    \[\Big\|\sum_{i=1}^k \lambda_i s_i^* \Big\|= \Big(\sum_{i=1}^k \lambda_i^2\Big)^{1/2}.\]
    \item[b)] In particular, every sequence $\{s_i^*\}_{i\in \N}$ of disjoint infinite segments in $JT^*$ is isometric to the usual basis of $\ell_2$.
\end{enumerate}
\end{lemma}
\begin{proof}
a) Let $\{\lambda_i\}_{i=1}^k$ be such that $(\sum_{i=1}^k \lambda_i^2)^{1/2} = 1$. Then $\sum_{i=1}^k \lambda_i s_i^*\in W$, hence $\|\sum_{i=1}^k \lambda_i s_i^*\| \leq 1$. On the opposite, let $\{n_i\}_{i=1}^k$ be incomparable nodes with $n_i \in s_i$, $i= 1 \dots k$. The definition of the norm of $JT$ yields $\|\sum_{i=1}^k \lambda_i e_{n_i}\| = 1$. Hence
\[ \Big\| \sum_{i=1}^k \lambda_i s_i^*\Big\| \geq \Big| \sum_{i=1}^k \lambda_i s_i^*\Big( \sum_{j=1}^k \lambda_j e_{n_j} \Big) \Big| = 1. \]
This ends the proof of the first part.

b) For each $i\in \N$, there exists a unique branch $\sigma_i$ such that $s_i$ is a final segment of $\sigma_i$. Thus, for every finite subset $F$ of $\N$, there exists $\{n_i\}_{i \in F}$ incomparable nodes of $\Tree$ with $n_i \in s_i$, and part a) implies that $\{s_i^* \}_{i\in F}$ is isometric to the usual basis of $\ell_2^{\#F}$.
\end{proof}

\begin{remark}
Let us observe for later use that given sequence $\{\sigma_i\}_i$ of branches of $\Tree$, there exists a sequence $\{n_i\}_i$ such that the sequence $\{\sigma_{i, > n_i}\}_i$ consists of pairwise disjoint final segments. Indeed, we set $n_1 = 1$ and for $i > 1$ we choose $n_i \in \s_i \setminus \cup_{j < i} \s_j$. Note that the final segments $\{ \sigma_{i, > n_i}\}_i$ are not necessarily incomparable.
\end{remark}

The next theorem is proved by J. Lindenstrauss and C. Stegall in \cite{LS}. Our proof is based on Proposition \ref{P3} which in turn is a consequence of Theorem \ref{A-I}. 
We fix the following notation for the quotient map
\begin{equation}
\label{quotient map notation}
V:JT^*\to JT^*/\B,
\end{equation}
as it will appear several times in the sequel.

\begin{theorem}\label{P2}
Let $V: JT^* \to JT^*/\B$ be the quotient map as in \eqref{quotient map notation}. Then the following holds:
\begin{enumerate}
    \item[a)] The family $\{V(\sigma^*): \sigma \in \Gamma\}$ norm-generates $JT^*/\B$ and it is isometrically equivalent to the usual basis of $\ell_2(\Gamma)$. Hence, $JT^*/\B$ is isometric to a Hilbert space denoted $\ell_2(\Gamma^*)$.
    \item[b)] For every separable closed subspace $X$ of $JT^*/\B$, there exists a closed subspace $Y$ of $JT^*$ such that the quotient map $V$ is an isometry between $Y$ and $X$.
\end{enumerate}
\end{theorem}
\begin{proof}
a) Since $JT^*$ is generated by the family $D = \{ e_n^* :n\in \N\} \cup \{ \sigma^*: \sigma \in \Gamma\}$ (Proposition \ref{P3}), it follows that $\{V(\sigma^*): \sigma \in \Gamma\}$ norm-generates $JT^*/\B$. For $\{\sigma_i\}_{i=1}^k \subset \Gamma$, since $\{e_n\}_n$ is bimonotone we get
\[\Big\| V\Big( \sum_{i=1}^k \lambda_i \sigma_i^* \Big) \Big\| = \lim_{n\to \infty} \Big\| \sum_{i=1}^k \lambda_i \sigma_{i,>n}^* \Big\|. \]
We recall that $\sigma_{>n}$ denotes the final segment $\{m\in \sigma : m>n\}$. Lemma \ref{L2} yields that there exists an $n_0 \in \N$ such that for $n > n_0$
\[ \Big\| \sum_{i=1}^k \lambda_i \sigma_{i,>n}^* \Big\| =\Big( \sum_{i=1}^k \lambda_i^2 \Big)^{1/2}. \]
Hence $\{V(\sigma_i^*)\}_{i=1}^k$ is isometric to the $\ell_2^k$ basis, and this proves part (a).

b) Let $\{\sigma_i\}_{i\in \N}$ be a sequence of branches. Then there exists a sequence $\{n_i\}_{i\in \N}$ such that $\{\sigma_{i, >n_i}\}_{i\in \N}$ are pairwise disjoint. Lemma \ref{L2} (b) yields that $\{\sigma_{i, >n_i}^*\}_{i\in \N}$ is isometric to the $\ell_2$ basis and since $V(\sigma_{>n}^*) = V(\sigma^*)$ for all $\sigma$ and $n$, we conclude that
\[ V: \Big[ \{\sigma^*_{i, >n_i}\}_{i\in \N} \Big] \to \Big[ \{V(\sigma_{i}^*)\}_{i\in \N} \Big] \]
is an isometry. Every closed separable subspace $X$ of $JT^*/\B$ is included in a subspace of the form $\big[\{V(\sigma_i^*)\}_{i\in \N}\big]$, and hence the proof is complete.
\end{proof}

\subsection{The Space $JT^{**}$}
For $\sigma \in \Gamma$, the sequence $\{e_n\}_{n\in\sigma}$ is a non-trivial weak Cauchy sequence and $w^*$-$\lim_{n\in \sigma} e_n = \sigma^{**}$. It is easy to check that $\|\sigma^{**}\| = 1$, $\sigma^{**}(e_n^*) = 0$ and $\sigma^{**}(\tau_{>n}^*) = \delta_{\sigma, \tau}$ for all $n \in \N$.

The space $JT^{**}$ coincides with $\B^{***}$, hence $JT^{**} = JT \oplus \B^\perp$ where $\B$ is viewed as a subspace of $JT^*$. The space $\B^\perp$ is isometric to the dual of $JT^*/ \B$, hence it is isometric to the space $\ell_2(\Gamma^{**})$ where $\Gamma^{**} = \{ \sigma^{**} : \sigma \in \Gamma\}$. Therefore $JT^{**} = JT \oplus \ell_2(\Gamma^{**})$.


\begin{proposition}\label{P4}
Let $Z$ be a closed separable subspace of $\ell_2(\Gamma^{**})$. Then $Z_\perp$ is a complemented subspace of $JT^*$.
\end{proposition}
\begin{proof}
Each $z^{**} \in Z$ is of the form
\[ z^{**} = \sum_{\sigma \in F} \lambda_\sigma \sigma^{**} = \sum_{\sigma \in F} z^{**}(\sigma^*) \sigma^{**} \]
with $F$ at most countable. The set $G= \{ \sigma : \exists z^{**} \in Z \text{ with } z^{**}(\sigma^*) \neq 0\}$ is countable because $Z$ is separable.
If $G = \{ \sigma_i \}_{i\in \N}$ and we choose $\{n_i\}_{i\in \N}$ such that $\{\sigma_{i, >n_i}\}_{i\in \N}$ are pairwise disjoint final segments, Lemma \ref{L2} (b) yields that the space $Y = [\{\sigma_{i, >n_i}\}_{i\in \N}]$ is isometric to $\ell_2$.

Let $\{z^{**}_k\}$ be an orthonormal basis of $Z$. Then $z^{**}_k = \sum_{i=1}^\infty \lambda_i^k \sigma_i^{**}$, $k \in \N$. We set $z_k^* = \sum_{i=1}^\infty \lambda_{i}^k \sigma_{i, > n_i}^*$.
Then $\{z_k^*\}_{k\in \N}$ and $\{z_k^{**}\}_{k\in \N}$ are biorthogonal, and both of them isometrically equivalent to the $\ell_2$ basis. Hence the expression
\[ P(x^*) = \sum_{k=1}^\infty z_k^{**}(x^*) z_k^* \]
defines a projection $P$ on $JT^*$, and clearly $\ker P = Z_\perp$.
\end{proof}

Recall that every vector in $JT^*$ can be represented by a sequence of scalars $(x_k)$ and the norm is given by Formula (\ref{normJT*}).

\begin{proposition}\label{F-G}
Let $(z_k)$ be a seminormalized sequence in $JT^*$.
Suppose that every segment $S$ in $\Tree$ intersects at most once $\supp(z_k)$ and $S\cap \supp(z_k)$ is always a segment. Then $(z_k)$ is equivalent to the unit vector basis of $\ell_2$.
\end{proposition}
\begin{proof}
Let $(x_k)$ in $JT$ with $\supp(x_k)\subset \supp(z_k)$ for each $k\in\N$. From the definition of $\|\cdot\|_{JT}$ it follows that $\|x_1+\cdots+x_m\|_{JT}^2= \|x_1\|_{JT}^2+ \cdots+\|x_m\|_{JT}^2$ for every $m\in\N$.

Let $n\in\N$ and let $\e>0$.
For each $k=1,\ldots,n$, we select $x_k\in JT$ such that $\|x_k\|_{JT}=1$, $\supp(x_k) \subset \supp(z_k)$ and $\|z_k\|^2\leq \langle x_k, z_k\rangle^2 +\e/n$.

If $x=\sum_{k=1}^n \langle x_k, z_k\rangle x_k$, then $\|x\|_{JT}^2= \sum_{k=1}^n \langle x_k, z_k\rangle^2$. Therefore
$$
\|x\|_{JT}^2\, \Big\|\sum_{l=1}^n z_l\Big\|^2 \geq  \Big|\Big\langle x, \sum_{l=1}^n z_l\Big\rangle\Big|^2 = \Big|\Big\langle \sum_{k=1}^n \langle x_k, z_k\rangle x_k, \sum_{l=1}^n z_l\Big\rangle\Big|^2.
$$

Since $\langle x_k,z_l\rangle=0$ for $k\neq l$, we get
$$
\|x\|_{JT}^2\, \Big\|\sum_{l=1}^n z_l\Big\|^2 \geq
\Big| \sum_{k=1}^n \langle x_k, z_k\rangle^2\Big|^2 \geq \|x\|_{JT}^2 \Big(\sum_{k=1}^n \|z_k\|^2-\e\Big).
$$
Hence $\|\sum_{k=1}^n z_k\|^2 \geq \sum_{k=1}^n \|z_k\|^2$.
\medskip

For the converse inequality, let $x\in JT$ and let us denote $x_k= P_{\supp(z_k)}x$, where $P_{\supp(z_k)}$ is the projection considered in Proposition \ref{P_A}. Then
\begin{eqnarray*}
\Big\langle x,\sum_{k=1}^n z_k\Big\rangle &=& \sum_{k=1}^n\langle x, z_k\rangle =
\sum_{k=1}^n\langle x_k, z_k\rangle \leq
\sum_{k=1}^n\|x_k\|_{JT}\, \|z_k\| \\
&\leq& \Big(\sum_{k=1}^n\|x_k\|_{JT}^2\Big)^{1/2} \Big(\sum_{k=1}^n\|z_k\|^2\Big)^{1/2}\leq \|x\|_{JT} \Big(\sum_{k=1}^n\|z_k\|^2\Big)^{1/2}.
\end{eqnarray*}
Hence $\|\sum_{k=1}^n z_k\|^2 \leq \sum_{k=1}^n \|z_k\|^2$, and the result is proved.
\end{proof}

\section{Complemented copies of $\ell_2$ and projectivity}\label{sect:l2}
In this section, we prove that every infinite-dimensional closed subspace of $\mathcal{B}, JT, JT^*$ contains a complemented subspace isomorphic to $\ell_2$. We use this result to establish the sub(super)projectivity for these spaces.

Given $0\neq x=(x_t)_{t\in \Tree}$ in $\B$ or $JT$, we denote $l(x)= \min\{|t| : t\in \supp(x)\}$ and  $u(x) = \sup\{|t| : t\in\supp(x) \}$.

\begin{definition} \label{definition1}
Let $X$ be a Banach space.
\begin{enumerate}
    \item The space $X$ is subprojective if every infinite-dimensional closed subspace $Y$ contains an infinite-dimensional closed subspace complemented in $X$.
    \item The space $X$ is superprojective if every closed subspace $Y$ of infinite codimension is contained into a subspace $Z$ of infinite codimension and complemented in $X$.
\end{enumerate}
\end{definition}

Subprojective and superprojective spaces were introduced in \cite{Whitley:64}. We refer to \cite{O-S:15} and \cite{G-P:16} for examples of these spaces.

\subsection{The spaces $\B$, $JT$ and $JT^*$ are subprojective}\label{sect:subp}

\begin{definition}
We say that a sequence $(x_k)$ of non-zero vectors in $\B$ or $JT$ is a \emph{level-block sequence} if $u(x_k)< l(x_{k+1})$ for each $k\in\N$.
\end{definition}

The following result is a direct consequence of the definition of the norm $\|\cdot\|_{JT}$.

\begin{lemma}\label{l-est}
Every level-block sequence $(x_k)$ in $JT$ admits a lower $\ell_2$-estimate. In fact, $\|x_1+\cdots+x_k\|^2_{JT}\geq \|x_1\|^2_{JT}+ \cdots+ \|x_k\|^2_{JT}$ for every $k\in\N$.
\end{lemma}
\begin{proof}
For every $k\in\N$, there exists a finite family of segments $G_k$ such that the supremum in Formula (\ref{JT-norm}) for $x_k$ is attained on this family. We can assume that for each node $t$ in the segments of $G_k$ we have $l(x_k)\leq |t|\leq u(x_k)$. Thus, considering $G_1\cup\cdots\cup G_k$ we obtain $\|x_1+\cdots+x_k\|^2_{JT}\geq \|x_1\|^2_{JT}+\cdots+ \|x_k\|^2_{JT}$.
\end{proof}

\begin{corollary}\label{u-est}
Every normalized level-block sequence $(v_k)$ in $\B$ admits an upper $\ell_2$-estimate.
\end{corollary}
\begin{proof}
Let $(a_k)\subset \R$ such that $\sum_{i=1}^\infty a_i v_i\in\B$ and let $x\in JT$ be a norm-one vector norming $\sum_{i=1}^\infty a_i v_i$.
By Proposition \ref{P_A}, we can assume that there exists a level-block sequence $(x_k)$ in $JT$ such that $l(v_k)\leq l(x_k)$ and $u(x_k)\leq u(v_k)$ for each $k\in\N$ and $x=\sum_{k=1}^\infty x_k$. In particular, $\langle v_k,x_k\rangle= \langle v_k,x\rangle$ for each $k\in\N$. Then
\begin{eqnarray*}
\Big\|\sum_{i=1}^\infty a_i v_i\Big\|_\B &=&
\sum_{i=1}^\infty a_i \langle v_i,x\rangle=
\sum_{i=1}^\infty a_i \langle v_i,x_i\rangle\\
&\leq&
\Big(\sum_{i=1}^\infty |a_i|^2\Big)^{1/2} \Big(\sum_{i=1}^\infty |\langle v_i,x_i \rangle|^2\Big)^{1/2}\\
&\leq& \Big(\sum_{i=1}^\infty |a_i|^2\Big)^{1/2} \Big(\sum_{i=1}^\infty
\| x_i\|_{JT}^2\Big)^{1/2}\\
&\leq& \Big(\sum_{i=1}^\infty |a_i|^2\Big)^{1/2} \|x\|_{JT} = \Big(\sum_{i=1}^\infty |a_i|^2\Big)^{1/2},
\end{eqnarray*}
where the last inequality is a consequence of Lemma \ref{l-est}.
\end{proof}

\begin{proposition}\label{l2inB}
Every normalized level-block sequence $(v_k)$ in $\B$ has a subsequence equivalent to the unit vector basis of $\ell_2$.
\end{proposition}
\begin{proof}
By Proposition \ref{P_A}, for every $k\in\N$ we can choose a norm-one vector $x_k\in JT$ such that $l(v_k)\leq l(x_k)$, $u(x_k)\leq u(v_k)$ and $\langle v_k,x_k\rangle=1$. Since $JT$ contains no copies of $\ell_1$, passing to a subsequence we can assume that $(x_k)$ is weakly Cauchy, hence $y_k = x_{2k}-x_{2k-1}$ is weakly null. By Theorem \ref{A-I}, passing to a subsequence we can assume that $(y_k)$ is equivalent to the unit vector basis of $\ell_2$.

Note that $\langle y_j,v_{2k}\rangle= \delta_{j,k}$. By Corollary \ref{u-est}, $(v_k)$ admits an upper $\ell_2$-estimate. Thus $T(a_i)=\sum_{i=1}^\infty a_i v_{2i}$ defines a continuous operator $T:\ell_2 \to \B$ and the conjugate operator $T^*:\B^*\to \ell_2$ satisfies
$$
\langle e_k,T^* y_l\rangle= \langle Te_k, y_l \rangle= \langle v_{2k},y_l\rangle= \delta_{k,l},
$$
where $(e_k)$ is the unit vector basis of $\ell_2$. Thus $T^* y_l=e_l$. Since $(y_l)$ is equivalent to $(e_l)$, $T^*$ is surjective. Therefore $T$ is an isomorphism (into), and hence $(v_{2i})$ is equivalent to the unit vector basis of $\ell_2$.
\end{proof}

The next result shows that the space $\B$ is subprojective.

\begin{theorem}\label{B-subp}
Every closed infinite-dimensional subspace $M$ of $\B$ contains a copy of $\ell_2$ complemented in $\B$.
\end{theorem}
\begin{proof}
A standard perturbation argument shows that given a sequence of positive numbers $(\e_k)$, there are a normalized level-block sequence $(u_k)$
in $\B$ and a sequence $(v_k)$ in $M$ such that $\|u_k-v_k\|<\e_k$ for each $k\in \N$. By Proposition \ref{l2inB}, passing to a subsequence
we can assume that $(u_k)$ is equivalent to the unit vector basis of $\ell_2$. Hence so is $(v_k)$.
Moreover, we can assume that the sequence $(x_k)\subset JT$ of coefficient functionals of $(v_k)$ is weakly Cauchy.
Then $(x_{2k}- x_{2k-1})$ is weakly null, and by Theorem \ref{A-I}, passing to a subsequence) we can assume that $(x_{2k}- x_{2k-1})$
is equivalent to the unit vector basis of $\ell_2$.

Observe that $Pu = \sum_{k=1}^\infty \langle u,(x_{2k}- x_{2k-1})\rangle v_{2k}$ defines a projection $P:\B\to\B$ onto $[v_{2k}]\subset M$.
Indeed, the map $P$ is well-defined because $(\langle u,(x_{2k}- x_{2k-1})\rangle)\in\ell_2$ for each $u\in [v_k]^*$,
and clearly $P^2=P$ and $P(\B)=[v_{2k}]$.
In particular, the subspace $[v_{2k}]$ is isomorphic to $\ell_2$ and complemented in $\B$.
\end{proof}


A version of the following result was stated without proof by Johnson and Rosenthal in \cite[Remark III.1]{J-R}.

\begin{theorem}\label{G-MA} \emph{(\cite[Theorem 1]{G-MA})}
If $(x_n^*)$ is a seminormalized sequence in the dual $X^*$ of a Banach space $X$ and $0$ is a weak$^*$-cluster point of $\{x^*_n : n \in \N\}$,
then there exist a basic subsequence $(y^*_n)$ of $(x^*_n)$ and a bounded sequence $(y_n)$ in $X$ such that $\la y_i, y^*_j\ra = \delta_{i,j}$.
\end{theorem}

Using this fact, we show that $JT$ and $JT^*$ are subprojective.

\begin{theorem}\label{JT-subp}
(1) Every closed infinite-dimensional subspace $M$ of $JT$ contains a copy of $\ell_2$ complemented in $JT$.

(2) Every closed infinite-dimensional subspace $N$ of $JT^*$ contains a copy of $\ell_2$ complemented in $JT^*$.
\end{theorem}
\begin{proof} 
(1) We proceed as in the proof of Theorem \ref{JT-subp}(1).
Given a seminormalized basic sequence $(x_k)$ in $JT$, passing to a subsequence we can assume that it is weakly Cauchy.
Thus $(y_k)= (x_{2k}-x_{2k-1})$ is a seminormalized weakly null sequence.
By Theorem \ref{G-MA}, we can find a subsequence of $(y_k)$ whose coefficient functionals $(v_k)$ can be chosen in the predual space $\B$.
By passing to a subsequence, we can assume that both sequences $(v_k)$ and $(y_k)$ are equivalent
to the unit vector basis of $\ell_2$.
Thus $Pu = \sum_{k=1}^\infty \langle u,y_k\rangle v_k$ defines a projection $P:\B\to\B$ onto $[v_k]$, and the conjugate operator
$P^*x= \sum_{k=1}^\infty \langle v_k,x\rangle y_k$ is a projection on $JT$ with range $[y_k]$. Note that $P^*$ is weak$^*$ continuous.

(2) Let $V:JT^*\to\ell_2(\Gamma)$ denote the quotient map from \eqref{quotient map notation}, with kernel $\B$, where
the quotient is identified isometrically with $\ell_2(\Gamma)$ via Theorem \ref{P2}. We consider two cases:

(i) If the restriction $V|_N$ is strictly singular, then a perturbation argument shows that there exists an automorphism $U$ on $JT^*$ and a closed infinite-dimensional subspace $N_1$ of $N$ such that $U(N_1)\subset \ker V =\B$. By Theorem \ref{B-subp}, $U(N_1)$ contains a subspace $N_2$ isomorphic to $\ell_2$ and complemented in $\B$. Since $N_2$ is reflexive, if $P$ is a projection on $\B$ onto $N_2$, then $P^{**}$ is a projection on $\B^{**}= JT^*$ with range $N_2$. Thus $U^{-1}(N_2)$ is a subspace isomorphic to $\ell_2$, contained in $N$ and complemented in $JT^*$.

(ii) If $V|_N$ is not strictly singular, then there exists a closed, separable, infinite-dimensional subspace $N_1$ of $N$ such that $V|_{N_1}$
is an isomorphism. Then $N_1\simeq \ell_2$ and $V(N_1)$ is complemented in $\ell_2(\Gamma)$, hence $N_1$ is complemented in $JT^*$.
Indeed, $JT^* = N_1\oplus V^{-1}(V(N_1)^\perp)$.
\end{proof}

Let us observe, for later use, that the following holds.

\begin{lemma}\label{compl.JT} Let $(x_n)_n$ be a level block sequence in $JT$ equivalent to the basis of $\ell_2$. Then the space $[(x_n)_n]$ is complemented in $JT$.
\end{lemma}
\begin{proof} For every $n\in \N$ select $x_n^*$ such that $\|x_n^*\|_{JT^*} \le C$ , $x_n^*(x_m )= \delta _{n,m} $ and $(x_n^*) _n$ is a level block sequence. The proof of Theorem \ref{l2inB} yields that   $(x_n^*) _n$  is equivalent to the basis of $\ell_2$. Hence $P(x) = \sum_{n}x_n^*(x)x_n$ defines a projection onto $[(x_n)_n]$.
\end{proof}

\begin{corollary}\label{higher-d}
All successive dual spaces of $\B$ satisfy Theorem \ref{B-subp}.
\end{corollary}
\begin{proof}
Theorem \ref{B-subp} is satisfied by $JT^{**}\simeq JT\oplus \ell_2(\Gamma)$ and $JT^{***}\simeq JT^*\oplus \ell_2(\Gamma)$, and for every positive integer $k$, the dual spaces of order $2k+1$ of $\B$ are isomorphic to $JT\oplus \ell_2(\Gamma)$ and the dual spaces of order $2k+2$ of $\B$ are isomorphic to $JT^*\oplus \ell_2(\Gamma)$ \cite[Corollary 3.c.7]{FetterG}.
\end{proof}

\subsection{The spaces $\B$, $JT$ and $JT^*$ are superprojective}\label{sect:superp}

\begin{proposition}\label{Bsuper}
The space $\B$ is superprojective.
\end{proposition}
\begin{proof}
Let $M$ be a closed infinite-codimensional subspace of $\B$. Then $M^\perp$ is an infinite dimensional subspace of $JT$.
As in the proof of Theorem \ref{B-subp}, given a sequence of positive numbers $(\e_k)$, there are a normalized level-block sequence $(z^*_k)$
in $JT$ and a sequence $(y^*_k)$ in $M^\perp$ such that $\|z^*_k-y^*_k\|<\e_k$ for each $k\in \N$.
By Theorem \ref{G-MA}, passing to a subsequence we can assume that coefficient functionals $(z_k)$ of  $(z^*_k)$ are in $\B$.
And by Proposition \ref{l2inB} and Theorem \ref{A-I}, passing to a subsequence we can assume that the sequences $(z_k)$ and $(z^*_k)$ are both
equivalent to the unit vector basis of $\ell_2$.

Taking the numbers $(\e_k)$ small enough, $Kx=\sum_{k=1}^\infty \langle x,z^*_k-y^*_k\rangle z_k$ defines an operator on $\B$ with $\|K\|<1$.
Then $U = I-K$ is an isomorphism in $\B$ and the conjugate operator $U^*:JT\to JT$ is given by
$U^*x^*= x^*-\sum_{k=1}^\infty \langle z_k,x^*\rangle(z^*_k-y^*_k)$ and satisfies $U^*z^*_k=y^*_k$ for every $k\in \N$.

We denote $y_k=U^{-1}z_k$. Then $(y_k)$ is equivalent to the unit vector basis of $\ell_2$ and
$$
\langle y_k,y^*_l \rangle= \langle U^{-1}z_k,U^*z^*_l \rangle= \langle z_k,z^*_l \rangle = \delta_{i,j}.
$$

Then $Px= \sum_{i=1}^\infty \langle x,y^*_i\rangle y_i$ defines a projection on $\B$ with kernel $\cap_{i=1}^\infty \ker z^*_i$,
which is an infinite codimensional subspace of $\B$ and contains $M$ because $(z^*_i)\subset M^\perp$.
Thus $\B$ is superprojective.
\end{proof}

To study $JT$, we need the following well-known result.

\begin{lemma}\label{nonsepdual}
If $X$ is separable and $X^*$ is non-separable, then every reflexive subspace of $X^*$ is separable.
\end{lemma}

It follows from Lemma \ref{nonsepdual} that if a subspace of $JT^*$ containing $\B$ is complemented, then the complement is isomorphic to a separable Hilbert space.

\begin{proposition}\label{JTsuper}
The space $JT$ is superprojective.
\end{proposition}
\begin{proof}
Let $M$ be a closed infinite-codimensional subspace of $JT$. Then $M^\perp$ is infinite-dimensional in $JT^*$. We consider three cases:
\smallskip

(1) \emph{$M^\perp\cap \B$ infinite-dimensional.} By Theorem \ref{B-subp}, $M^\perp\cap \B$ contains a subspace $N$ isomorphic to $\ell_2$ and complemented in $\B$. If $P\in\Lc(\B)$ is a projection on $\B$ onto $N$, then $P^{**}$ is a projection on $JT^*$ onto $N$. Thus $N$ is complemented in $JT^*$.

If $A$ and $C$ denote $N$ as a subspace of $\B$ and $JT^*$, then $A^\perp =C_\perp$ and $C= (C_\perp)^\perp$. Hence $A^\perp$ is an infinite-codimensional complemented subspace of $JT$ containing $M$.

(2) \emph{$M^\perp\cap \B$ finite-dimensional and $M^\perp +\B$ non-closed.} A perturbative argument allows us to reduce this case to the first one.

Indeed, in this case the range of the operator $Q_BJ_{M^\perp}$ is $(M^\perp+B)/B$, which is not closed in $JT/B$. By \cite[Proposition 2.c.4]{L-T:77},
there exists a closed infinite-dimensional subspace $M_1$ of $M^\perp$ such that $Q_\B J_{M_1}$ is compact.
So there is a normalized weakly null sequence $(x^*_n)$ in $M^\perp$ with $\textrm{dist}(x^*_n,\B)\to 0$. By Theorem \ref{G-MA}, passing to a subsequence
we can assume the existence of $(y^*_n)$ in $\B$ and a bounded sequence $(x_n)$ in $JT$ so that $\la x_i,x^*_j\ra= \delta_{i,j}$ and
$\sum_{n=1}^\infty \|x_n\|\, \|x^*_n- y^*_n\|<1$.

The operator $K\in\Lc(JT)$ defined by $K(x)= \sum_{n=1}^\infty \la x,x^*_n-y^*_n\ra x_n$ has norm $\|K\|<1$ and $K^*(x^*)= \sum_{n=1}^\infty \la x_n,x^*\ra(x^*_n-y^*_n)$. Thus $I-K$ is bijective and $(I-K^*)x^*_n= y^*_n$.

We take $N=(I-K)^{-1}(M)$. Then $N^\perp= (I-K^*)(M^\perp)$. Thus $N^\perp\cap \B$ is infinite-dimensional. By (1), $N$ is contained in an infinite-codimensional complemented subspace $U$. Hence $N=(I-K)^{-1}(U)$ is an infinite-codimensional complemented subspace containing $M$.
\smallskip

(3) \emph{$M^\perp\cap\B$ finite-dimensional and $M^\perp +\B$ closed.} Adding to $M$ a finite-dimensional subspace, we can assume $M^\perp\cap \B=\{0\}$. Then the restriction $V|_{M^\perp}$ is an isomorphism. Hence, since $V(M^\perp)$ is complemented, $M^\perp$ is complemented in $JT^*$, and it is isomorphic to $\ell_2$ by Lemma \ref{nonsepdual}.

We take an orthonormal basis $(z_n)$ in $M^\perp$ and, as in the proof of Theorem \ref{JT-subp}(1), we select a bounded sequence $(x_n)$ in $JT$ such that $\la x_i, z_j\ra=\delta_{i,j}$. Passing to a subsequence we can assume that $(x_n)$ is weakly Cauchy, hence $y_n= x_{2n}- x_{2n-1}$ provides a weakly null sequence such that $\la y_i, z_{2j}\ra=\delta_{i,j}$. By Theorem \ref{A-I}, passing to a subsequence we can assume that $(y_i)$ is equivalent to the unit vector basis of $\ell_2$.

Thus we can assume that both $(z_n)$ in $M^\perp$ and $(y_i)$ in $JT$ are equivalent to the unit vector basis of $\ell_2$ and biorthogonal. Hence $P(x) =\sum_{n=1}^\infty \la x,z_n\ra y_n$ defines a projection on $JT$ with range $[y_n]$ and kernel $\{z_n : n\in\N\}_\perp \supset M$, and the result is proved.
\end{proof}

Next, we prove that $JT^*$ is superprojective using the following consequence of the results of \cite{G-P:24}.

\begin{proposition}\label{GP-24}
Let $M$ be a closed subspace of $X$. If $M$ is superprojective and every closed infinite-codimensional subspace of $X$ containing $M$ is contained in a complemented infinite-codimensional subspace, then $X$ is superprojective.
\end{proposition}
\begin{proof}
With the terminology of \cite{G-P:24}, $M$ superprojective implies $J_M$ superprojective, and the second condition is equivalent to $Q_M$ superprojective. So, the result follows from \cite[Theorem 2.7]{G-P:24}.
\end{proof}

\begin{proposition}
The space $JT^*$ is superprojective.
\end{proposition}
\begin{proof}
By Proposition \ref{GP-24}, we have to show that every closed infinite-codimensional subspace $M$ of $JT^*$ containing $\B$ is contained in an
infinite-codimensional complemented subspace.
Note that $JT=\B^*$ and $JT^*/\B\simeq \ell_2(\Gamma)$, hence $JT^{**}\simeq JT\oplus \ell_2(\Gamma^*)$ and $\B^\perp =\ell_2(\Gamma^*)$.

Since $M^\perp$ is a closed infinite-dimensional subspace of $\B^\perp=\ell_2(\Gamma^*)$, it contains  an infinite-dimensional separable subspace $Z$.
Then $Z_\perp$ is an infinite-codimensional subspace of $JT^*$ that contains $M$, and $Z_\perp$ is complemented in $JT^*$ by Proposition \ref{P4}.
\end{proof}

\begin{corollary}
All the successive dual spaces of $\B$ are superprojective.
\end{corollary}
\begin{proof}
It is similar to that of Corollary \ref{higher-d}.
\end{proof}

\subsection{Uncomplemented copies of $\ell_2$}\label{sect:uncompl}

For $n\in\N$ and $1\leq p\leq\infty$, we denote $\ell_p^n =(\R^n,\|\cdot\|_p)$. We will need the following result.

\begin{proposition}\label{prop:FG} \emph{\cite[Theorem 2.d.6]{FetterG}}
Let $(n_i)$ be a sequence in $\N$. Then $\ell_2( \ell_\infty^{n_i})$ is isomorphic to a complemented subspace of $J$.
Consequently, $\ell_2(\ell_1^{n_i})$ is isomorphic to a complemented subspace of $J^*$.
\end{proposition}

If $X$ and $Y$ are isomorphic Banach spaces, the \emph{Banach-Mazur distance between $X$ and $Y$} is defined by
$$
d_{BM}(X,Y) = \inf\{\|T^{-1}\| : T\in \Lc(X,Y)  \textrm{ bijective and }\|T\|=1\}.
$$

For a finite-dimensional space $Y$, a constant $C\geq 1$, and a Banach space $X$ containing a subspace $N_0$
with $d_{BM}(Y,N_0) \leq C$, the $C$-projection constant of $Y$ in $X$ is
$$
A_C(Y,X)=\inf\{\|P\|: P \textrm{ projection on $X$ onto $N$ with }d_{BM}(Y,N) \leq C\}.
$$

A Banach space $X$ has the \emph{Dunford-Pettis property} if every weakly compact operator $T:X\to Y$ takes weakly convergent
sequences to convergent sequences.

It is easy to prove that complemented reflexive subspaces of Banach spaces with the Dunford-Pettis property
are finite dimensional.

\begin{lemma}\label{l2-linfty}
There exists $C>1$ and a sequence $(k_n)$ in $\N$ such that for each $n\in \N$ there exists a subspace $Y_n$ of $\ell_\infty^{k_n}$ with
$d_{BM}(Y_n,\ell_2^n)\leq C$ and $\lim_{n\to\infty} A_C(Y_n,\ell_\infty^{k_n}) =\infty$.
\end{lemma}
\begin{proof}
The existence of the subspaces $Y_n$ follows from the fact that $\ell_2$ is finitely representable in $c_0$.
Moreover $\lim_{n\to\infty} A_C(Y_n, \ell_\infty^{k_n}) =\infty$. Otherwise, taking ultrapowers with respect to a non-trivial ultrafilter $\U$ on $\N$,
we get that $(\ell_\infty)_\U$, which is a $\Lc_\infty$-space, has a complemented subspace isomorphic to an infinite-dimensional Hilbert space,
which is not possible because $\Lc_\infty$-spaces have the Dunford-Pettis property \cite[Chapter 1]{Bo:81}.
\end{proof}

\begin{lemma}\label{l2-l1}
There exists $C>1$ such that for each $n\in\N$ there exists a subspace $X_n$ of $\ell_1^{2^n}$ such that $d_{BM}(X_n,\ell_2^n)\leq C$ and
$\lim_{n\to\infty} A_C(X_n,\ell_1^{2^n}) =\infty$.
\end{lemma}
\begin{proof}
We take as $X_n$ the subspace generated by the Rademacher functions $r_1,\ldots,r_n$ in $L_1(0,1)$, which is contained in a subspace
generated by $2^n$ pairwise disjoint characteristic functions, and $C=d_{BM}([r_n],\ell_2)$.
That $\lim_{n\to\infty} A_C(X_n, \ell_1^{2^n}) =\infty$ can be proved as in Lemma \ref{l2-linfty}, since $\Lc_1$-spaces also have the Dunford-Pettis property.
\end{proof}

We can see that Lemmas \ref{l2-linfty} and \ref{l2-l1} can be derived from Grothendieck's theorem in \cite[Section]{LP:68}.

\begin{proposition}
Each of the spaces $J$ and $J^*$ contains a copy of $\ell_2$ which is not complemented.
\end{proposition}
\begin{proof}
It is a consequence of Proposition \ref{prop:FG} and Lemmas \ref{l2-linfty} and \ref{l2-l1}.
\end{proof}

\begin{proposition}
Each of the spaces $\B$, $JT$ and $JT^*$ contains a copy of $\ell_2$ which is not complemented.
\end{proposition}
\begin{proof}
Given a branch $B\in\Gamma$, the subspace $[e_t : t\in B]$ is isomorphic to $J^*$ in $\B$ and $JT^*$, and it is isomorphic to $J$ in $JT$.
\end{proof}

\section{Non-Reflexive Subspaces of $JT$}

In this section, we will show that every non-reflexive subspace of $JT$ contains a copy of the James space $J$. For this, we need a refinement of the
corresponding result for $J$.

\subsection{Non-Trivial Weakly Cauchy Sequences in $J$}

A sequence of real numbers $x = (x(n))_{n\in \N}$ is in $J$ when
\begin{equation}\label{eq:Jnorm}
\Vert x \Vert= \sup \biggl\{\biggl( \sum_{i=1}^{n} \Big| \sum _{k \in I_{i} }x(k)\Big|^2 \biggr)^{1/2} :\; \{ I_i \}_{i=1}^n \text{ disjoint intervals of } \N \biggr \} <\infty.
\end{equation}

The basis $(e_n)_{n \in \N}$ of $c_{00}$ defines a boundedly complete basis of $J$. Further, every interval $I$ of $\N$ defines by the rule $I^*(x) = \sum_{n\in I} x(n)$ a norm-one functional $I^* \in J^*$. We denote by $I_{\infty}^*$ the functional corresponding to $I= \N$.
Notice that the basis of $J$ is isometric to the sequence $(e_n)_{n\in \s}$ in $JT$ for every $\s \in \Gamma$.

\begin{definition}
A sequence $(x_n)_n$ in a Banach space $X$ is called non-trivial weakly Cauchy if $x_n\xrightarrow{w^*} x^{**} \in X^{**} \setminus X$.
\end{definition}

We also need the following definition and result from \cite{AG}.

\begin{definition}[\cite{AG}, Definition 3.1]
Let $x\in J$. A family $\{I_i\}_{i\in F}$ of finite or infinite disjoint intervals of $\N$ with $F$ either an initial interval of $\N$ or $F=\N$ is said to be an $x$-norming partition if
\[ \| x\| = \Big( \sum_{i\in F} \Big( \sum_{k\in I_i} x(k) \Big)^2 \Big)^{1/2}. \]
\end{definition}

\begin{proposition}[\cite{AG}, Corollary 3.4]  \label{AG1}
Every $x \in J$ admits an $x$-norming partition.
\end{proposition}

The next result is known. We give a proof for the sake of completeness.

\begin{lemma}\label{L6}
Let $(x_n)_{n\in \N}$ be a block sequence in $J$ and $\epsilon > 0$. Assume that $I_{\infty}^*(x_n) = 0$ and there exists an $\alpha> 0$ such that $\big| \|x_n\| - \alpha\big| < \alpha \epsilon$ for all $n\in \N$. Then $(x_n)_{n\in \N}$ is equivalent to the usual basis of $\ell_2$. More precisely,
\begin{equation}\label{eqn: eq.2}
(1 - \epsilon) \alpha \Big(\sum_{n=1}^\infty \lambda_n^2\Big)^{1/2} \le \Big\|\sum_{n=1}^\infty \lambda_n x_n\Big\| \le \sqrt{2} (1 + \epsilon) \alpha \Big(\sum_{n=1}^\infty \lambda_n^2\Big)^{1/2}.
\end{equation}
\end{lemma}
\begin{proof}
For the first inequality, we choose a partition $\{I_i\}_{i\in F}$ such that for every $i\in F$ there exists a (unique) $n\in \N$ with $I_i \subset ran(x_n)$ and also for all $n\in \N$ the family $\{I_i : I_i \subset \ran(x_n)\}$ is an $x_n$-norming partition. Then for $\{\lambda_n\}_{n\in\N}$ in $\R$ we have
\begin{equation}\label{eqn: eq.3}
\Big\| \sum_{n=1}^\infty \lambda_n x_n \Big\| \geq \Big( \sum_{i \in F} I_i^* \Big( \sum_{n=1}^\infty \lambda_n x_n\Big)^2 \Big)^{1/2} > (1-\epsilon) \alpha \Big(\sum_{n=1}^\infty \lambda_n^2 \Big)^{1/2}.
\end{equation}
We pass to the second inequality. Since $I_{\infty}^*(x_n) = 0$ for all $n\in \N$, given $\{\lambda_n \}_{n\in \N}$ and an interval $I$,
\begin{equation}\label{eqn: eq.1}
\Big( I^*\Big( \sum_{n=1}^\infty \lambda_n x_n\Big) \Big)^2 \leq 2\Big( \lambda_{n_1}^2 I_1^*(x_{n_1})^2 + \lambda_{n_2}^2 I_2^*(x_{n_2})^2 \Big),
\end{equation}
where $n_1= \min\{ n : \supp(x_n) \cap I \neq \emptyset\}$, $n_2 = \max\{ n : \supp(x_n) \cap I \neq \emptyset\}$, $I_1 = \ran(x_{n_1}) \cap I$ and $I_2 = \ran(x_{n_2}) \cap I$.

Let $x = \sum_{n=1}^\infty \lambda_n x_n$ and let $\{ I_i\}_{i\in F}$ be an $x$-norming partition. By the previous observation, for every $I_i$ intersecting the support of more than one $x_n$, there are $I_i^1, I_i^2$ such that $I_i, I_i^1, I_i^2$ satisfy inequality (\ref{eqn: eq.1}) and each $I_i^1, I_i^2$ intersects the support of only one $x_n$.
Therefore there exists a partition $\{J_j\}_{j\in G}$ such that either $J_j = I_i$ for some $i\in F$ or $J_j = I_i^1$ or $J_j = I_i^2$ for some $i\in F$. For each $n\in \N$, we set
\[ G_n = \{ j\in G : J_j \cap \supp(x_n) \neq \emptyset \} \]
Clearly, $\{G_n\}_{n\in \N}$ are pairwise disjoint and from \ref{eqn: eq.1} we have:
\begin{equation}\label{eqn: eq.4}
\|x\|^2 =  \sum_{i\in F} \Big( \sum_{k\in I_i} x(k) \Big)^2 \leq 2 \sum_{n=1}^\infty \sum_{j\in G_n} \lambda_n^2(J_j^*(x_n))^2  \leq 2(1+ \epsilon)^2 \alpha^2 \sum_{n=1}^\infty \lambda_n^2.
\end{equation}
Inequalities \ref{eqn: eq.3} and \ref{eqn: eq.4} complete the proof.
\end{proof}

Applying a sliding hump argument and Lemma \ref{L6}, we derive the following result.

\begin{proposition}\label{P5}
Every seminormalized sequence $(x_n)_n$ in $J$ such that $I_\infty^*(x_n) \to 0$ and $x_n(k) \to 0$ for every $k \in \N$ has a subsequence $(x_{n_k})_k$ equivalent to the basis of $\ell_2$.

In particular, every weakly null $(x_n)_n$ in $J$ with $\lim_n \|x_n\| = \alpha > 0$ has an $\ell_2$-subsequence $(x_{n_k})_k$, and if for some $0 < \epsilon < 1$, we have
\begin{equation}\label{eqn:eq.21}
\sum_{k=1}^\infty \Big| I_\infty^*(x_{n_k})\Big| < \alpha \frac{\epsilon}{4} \quad \text{and} \quad
\Big| \|x_{n_k}\| - \alpha \Big| < \alpha \frac{\epsilon}{4}
\end{equation}
then
\begin{equation}\label{eqn:eq.5}
(1 -\epsilon) \alpha\Big(\sum_{k=1}^\infty \lambda_k^2\Big)^{1/2} \le \Big\|\sum_{k=1}^\infty \lambda_k x_{n_k}\Big\| \le (1+\epsilon)\sqrt{2} \alpha \Big(\sum_{k=1}^\infty \lambda_k^2\Big)^{1/2}.
\end{equation}
\end{proposition}
\begin{proof}
For every $k\in \N$, we choose $m_k \in \supp(x_{n_k})$ and we set $y_k = x_{n_k} - \alpha_{n_k} e_{m_k}$.
It is easy to check that
\[ I_\infty^*(y_{k}) = 0 \quad \lim_k \|y_{k}\| = \alpha \quad \text{and} \quad \big|\|y_k\| - \alpha\big| < \alpha \frac{\epsilon}{2} \]
Therefore, Lemma \ref{L6} yields
\begin{equation}\label{eqn: eq.19}
(1 - \frac{\epsilon}{2}) \alpha \Big(\sum_{k=1}^\infty \lambda_k^2\Big)^{1/2} \le \Big\|\sum_{k=1}^\infty \lambda_k y_k\Big\| \le \sqrt{2} \Big(1 +\frac{\epsilon}{2}\Big) \alpha \Big(\sum_{k=1}^\infty \lambda_k^2\Big)^{1/2}
\end{equation}
Since
$\sum_{k=1}^\infty \|x_{n_k} - y_k\| < \epsilon/4$,
the triangle inequality gives
\begin{equation}\label{eqn: eq.20}
(1 - \epsilon) \alpha \Big(\sum_{k=1}^\infty \lambda_k^2\Big)^{1/2} \le \Big\|\sum_{k=1}^\infty \lambda_k x_{n_k}\Big\| \le \sqrt{2} (1 + \epsilon) \alpha \Big(\sum_{k=1}^\infty \lambda_k^2\Big)^{1/2},
\end{equation}
and the proof is complete.
\end{proof}

We pass to study the non-trivial weakly Cauchy sequences in $J$. First we show that each of these sequences has a subsequence equivalent to the basis of $J$.

\begin{proposition}\label{P10}
Let $(x_n)_n$ be a block sequence in $B_J$ such that
\[ \lim_n I_\infty^*(x_n) = \alpha \neq 0 \quad \text{and} \quad \lim_n \|x_n\| =\beta \geq 0,\]
and let also $0< \epsilon < 1$. Then there exists a subsequence $(x_{n_k})_k$ such that for every $(\lambda_k)_k$ in $\R$ we have
\begin{equation}\label{eqn:eq.6}
(1-\epsilon) |\alpha|\Big\| \sum_{k=1}^\infty \lambda_k e_k \Big\| \leq \Big\| \sum_{k=1}^\infty \lambda_k x_{n_k} \Big\| \leq 3(1+\epsilon)(|\alpha|+ \beta)\Big\| \sum_{k=1}^\infty \lambda_k e_k \Big\|.
\end{equation}
Hence $(x_{n_k})_k$ is equivalent to the basis $(e_k)_k$ of $J$.
\end{proposition}
\begin{proof}
We assume that $(x_n)_n$ is a skipped block sequence, and for every $n\in \N$ we choose $m(n) \in \N$ with $\supp(x_n) < m(n) < \supp(x_{n+1})$, we set $v_n = x_n - \alpha e_{m(n)}$, and we assume that
$\lim_n \|v_n\| = \gamma \leq |\alpha| + \beta$.

We will show that a subsequence $(x_{n_k})_k$ such that
\begin{equation}\label{eqn:eq.7}
\sum_{k=1}^\infty \Big| I_\infty^*(x_{n_k}) - \alpha \Big| < |\alpha| \frac{\epsilon}{4} \quad \text{and} \quad
\Big| \|v_{n_k}\| - \gamma \Big| \leq \gamma\frac{\epsilon}{4} \leq (|\alpha| + \beta) \frac{\epsilon}{4}
\end{equation}
satisfies inequality (\ref{eqn:eq.6}). We start with the first part.

Let $(\lambda_k)_k$ be a sequence in $\R$ and fix disjoint intervals $I_1,\ldots,I_N$ of $\N$. For $1\leq m\leq N$, let
$L_m$ as the minimum interval such that $k\in I_m$, if $\supp(v_{n_k}) \subset L_m$. We have
\begin{align*}
\Big\|\sum_{k=1}^\infty\lambda_k x_{n_k}\Big\| &\geq \Big(\sum_{m=1}^N\Big( L_m^*\Big( \sum_{k=1}^\infty \lambda_k x_{n_k}\Big) \Big)^2\Big)^{1/2}\\
&\geq \Big(\sum_{m=1}^N\Big( L_m^* \Big(\sum_{k = 1}^\infty \lambda_k \alpha e_{m(n_k)}\Big)\Big)^2\Big)^{1/2}\\
&- \Big(\sum_{m=1}^N\Big(  L_m^*\Big(\sum_{k=1}^\infty \lambda_k (x_{n_k}-\alpha e_{m(n_k)}) \Big)\Big)^2\Big)^{1/2}\\
&\geq |\alpha|\Big(\sum_{m=1}^N\Big(\sum_{k\in I_m}\lambda_k\Big)^2\Big)^{1/2} - \sum_{k=1}^\infty|\lambda_k|\Big|I^*_\infty(x_{n_k}) - \alpha\Big|\\
&\geq |\alpha|\Big(\sum_{m=1}^N\Big(\sum_{k\in I_m}\lambda_k\Big)^2\Big)^{1/2} -\frac{\epsilon}{4}|\alpha|\|(\lambda_k)_k\|_\infty,
\end{align*}
Taking a supremum over all disjoint intervals $(I_m)_{m=1}^N$ of $\N$ implies
\[ \Big\| \sum_{k=1}^\infty \lambda_k x_{n_k}\Big\| \geq |\alpha| \Big\| \sum_{k=1}^\infty \lambda_k e_{k}\Big\|  - \frac{\epsilon}{4}|\alpha|\|(\lambda_k)_k\|_\infty\geq (1-\epsilon)|\alpha|\Big\| \sum_{k=1}^\infty \lambda_k e_{k}\Big\|. \]
We pass to the second part of inequality (\ref{eqn:eq.6}).
We first observe that the sequence $(v_{n_k})_k$ satisfies the assumptions of Proposition \ref{P5}, hence from (\ref{eqn:eq.5}) we get
\begin{equation}\label{eqn:eq.22}
\Big\|\sum_{k=1}^\infty \lambda_k v_{n_k}\Big\| \le (1+\epsilon)\sqrt{2} \gamma \Big(\sum_{k=1}^\infty \lambda_k^2\Big)^{1/2} \leq (1+\epsilon)\sqrt{2} (|\alpha| + \beta) \Big(\sum_{k=1}^\infty \lambda_k^2\Big)^{1/2}.
\end{equation}
Moreover,
\begin{equation}\label{eqn:eq.9}
\Big\| \sum_{k=1}^\infty \lambda_k \alpha e_{m(n_k)}\Big\| = |\alpha| \Big\| \sum_{k=1}^\infty \lambda_k e_k \Big\|.
\end{equation}

Since $x_{n_k} = v_{n_k} + \alpha_{n_k}e_{m(n_k)}$, inequalities (\ref{eqn:eq.22}) and (\ref{eqn:eq.9}), the triangle inequality, and the fact that the norm in the basis $(e_n)_n$ dominates the $\ell_2$ norm, we derive
\[ \Big\| \sum_{k=1}^\infty \lambda_k x_{n_k}\Big\| < 3(1+\epsilon)(|\alpha| + \beta) \Big\| \sum_{k=1}^\infty \lambda_k e_k \Big\|, \]
completing the proof of the second part of inequality (\ref{eqn:eq.6}).
\end{proof}

For the general result we need the following fact.

\begin{proposition}\label{P7}
Let $X$ be a Banach space with a boundedly complete basis $(e_n)_n$, and let $(z_n)_n$ be a sequence in $X$ with $z_n \xrightarrow{w^*} z^{**} \in B_{X^{**}}\setminus B_X$. Then there exists a subsequence $(z_{n_k})_k$ of $(z_n)_n$ equivalent to a block sequence $(y_k)_k$ of the basis $(e_n)_n$.

Moreover, if every subsequence of $(y_k)_{k}$ has a further subsequence generating a complemented subspace of $X$ then the subsequence $(z_{n_k})_k$ can be chosen to satisfy the same property.
\end{proposition}
\begin{proof}
Denoting by $(e_n^*)_n$ the sequence in $X^*$ of coefficient functionals of the basis of $X$, $X_* = [(e_n^*)_n]$ is the predual of $X$ because $(e_n)_n$ is boundedly complete.
This yields that $X^{**} = X\oplus X_*^\perp$. Hence $z^{**} = y + y^{**}$ with $y\in X$ and $y^{**} \in X_*^\perp$. If $y=0$, then for $m\in \N$ we have $\lim_n e_m^*(z_n) = 0$; therefore, by a sliding hump argument, we conclude that there is a subsequence $(z_{n_k})_k$ which is a small permutation of a block sequence $(y_k)_k$, which ends the proof.

If $y \neq 0$ then we set $x_n= z_n-y$ and observe that $\lim_n e_m^*(x_n) = 0$ for each $m\in \N$. Hence there is a subsequence $(x_{n_k})_k$ which is a small permutation of a block sequence $(y_k)_k$ and clearly the two sequences are equivalent. Also $y_k \xrightarrow{w^*} y^{**}$. We will show that there exists a $k_0$ such that $(z_{n_k})_{k\geq k_0}$ is equivalent to $(y_k)_{k \geq k_0}$, which ends the proof.

Indeed, choose $x^* \in X_*$ a finite linear combination of $(e_n^*)_n$ with $x^*(y) = 1$, and  $y^*\in B_{X^*}$ such that $y^{**}(y^*) > \alpha >0$. With no loss of generality, we can assume that for each $k\in \N$ we have $y^*(y_k) > \alpha$. Finally, denote
\[ k_0 = \min \big\{ k: \max\, \supp(x^*) < \min \supp(y_k) \big\}. \]
\begin{claim}
The sequences $(y_k)_{k\geq k_0}$ and $(y_k + y)_{k\geq k_0}$ are equivalent.
\end{claim}
Since $X= [\{y\}] \oplus \ker x^*$, there are numbers $c, C > 0$ such that for $x= \lambda y + z$
\[ c \max\{ |\lambda| \|y\|, \|z\| \} \leq \|x\| = \| \lambda y + z \| \leq C \max\{ |\lambda| \|y\|, \|z\| \}. \]
Then $v= \sum_{k=k_0}^m\lambda_k y_k$ and $w =\sum_{k=k_0}^m \lambda_k (y_k + y)$ satisfy $c \|v\| \leq \|w\|$.

In the opposite direction,
\[\|w\| \leq C \max \Big\{\Big|\sum_{k=k_0}^m \lambda_k\Big| \|y\|, \|v\| \Big\}. \]
Since $\|v\| \geq |y^*(v)| \geq |\sum_{k=k_0}^m \lambda_k| \alpha$, we conclude that there exists a $C_1 > 0$ such that $\|w\| \leq C_1 \|v\|$.
This completes the proof of the claim.

Since $(y_k)_{k\geq k_0}$ is a small permutation of $(x_{n_k})_{k\geq k_0}$, $(y_k + y)_{k\geq k_0}$ is a small permutation of $(z_{n_k})_{k\geq k_0}$, and the claim implies the equivalence of $(y_k)_{k\geq k_0}$ and $(z_{n_k})_{k\geq k_0}$.

Next we prove the second part of the statement of the proposition. As before, $z^{**} = y+ y^{**}$. If we assume that $y = 0$, as in the first part of the proof, there is a subsequence $(z_{n_k})_k$ which is a small permutation of a block sequence $(y_k)_k$ and if every subsequence of $(y_k)_k$  generates a complemented subspace of $X$ then Proposition 1.a.9 in \cite{L-T:77} completes the proof.

Assume next that $y\ne 0$. Let $x^*$ and $ k_0 $ as before with $x^*(z) = 1$, $x^*\in [ (e_l^*)_{l\le {m}}]$ and $m< \min\supp y_k$ for $k>k_0$.

\begin{claim} For every subsequence of $(y_k) _{k>k_{0}}$ the space generated by $ (y+y_k)_{k\in L}$
 is complemented in $X$.
\end{claim}

Indeed, set $W = \ker x^*$. Clearly $(y_k)_{k\in L} \subset W$ and let  $P_{W,L} $ be the restriction of the projection $P_L:  X\to [(y_k)_{k\in L}]$  to $W$. Then it is easy to see that the operator
\[ Q: X=<y> \oplus Kerx^* \to [(y+y_k)_{k\in L}] \] defined by $Q(x_1 +x_2 ) = x^*(x_1 )y+P_{W,L}(x_2)$ is a bounded projection.

Since $\| z_{n_k} -( y+y_k)\| \to 0$, Proposition 1.a.9 in \cite{L-T:77} yields the result and the proof is complete.
\end{proof}
\begin{lemma}\label{compl.J} Let $(x_n)_n$ be a block sequence in $J$ equivalent to the basis of $J$. Then the space generated by some subsequence of $(x_n)_n$ is complemented in $J$.
\end{lemma}
\begin{proof} Since $(x_n)_n$ is equivalent to the basis of $J$, by Proposition \ref{P5}, we may, by passing to a subsequence and applying a small perturbation, assume that for all $n\in \N$
$I_n^*(x_n) =1$, where $I_n^* = \ran x_n$ is the smallest interval containing $\supp x_n$. We consider a complete partition $(J_n)_n$ of $\N$
(i.e. $\N= \cup_n J_n$) into segments  such that $I_n\subset J_n$.
\begin{claim} The projection $ P(x) = \sum_n J_n^*(x) x_n$ is bounded.
\end{claim}
Let $x\in J$ then there exists $C> 0$ such that
\begin{equation}\label{eqn: eq 40}
 \|P(x)\|= \| \sum_n J_n^*(x) x_n\| \le C\| \sum_n J_n^*(x) e_n\|.
 \end{equation}
We set $y = \sum_n  J_n^*(x) e_n$ and we choose $Q=\{ L_i \}_{i\in F}$ an $y$-norming partition. Let $Q_1= \{ M_i \}_{i\in F}$ with $ M_i = \cup_{n\in {L_i} }J_n$. Then
\[\| P(x)\| \le C \| y\| = C \| y\|_Q \quad \text{and }\quad \|y \|_Q = \|x\|_{Q_1} \le \|x\| \]
We recall that $\| y \|_Q$ denotes the $Q$ norm of $y$  as it is defined in (\ref{JT-P}) and the same for
$\|x\|_{Q_1} $.

Hence $\|P(x)\| \le C \|x\|$.
 This  ends the proof of the claim and completes the proof of the lemma.
\end{proof}

Lemma \ref{compl.J} and Propositions \ref{P10},  \ref{P7} yield the following result.

\begin{theorem}\label{th:J-seq}
Every non-trivial weakly Cauchy sequence $(x_n)_n$ in $J$ admits a subsequence $(x_{n_k})_k$ equivalent to the basis of $J$ such that $[x_{n_k}]_k$ is complemented in $J$.
\end{theorem}

\begin{corollary}
Every seminormalized basic sequence $(x_n)_n$ in $J$ has a subsequence $(x_{n_k})_k$ equivalent to the basis of $\ell_2$ or to the basis of $J$ with $[x_{n_k}]_k$ complemented in $J$.
\end{corollary}
\begin{proof} The sequence $(x_n)_n$ has a subsequence which either weakly null or non trivial weakly Cauchy. In the first case the result follows from Lemma \ref{compl.JT} and Proposition
\ref{P7} by considering $J$ as a subspace of $JT$ in a natural manner while in the second one it follows from Lemma \ref{compl.J} and Proposition \ref{P7}.
\end{proof}

We close this subsection with an immediate consequence of Theorem \ref{th:J-seq}.

\begin{corollary}\label{TH_1}
Every non-reflexive subspace of $J$ contains isomorphically the James space $J$ as a complemented subspace of $J$.
\end{corollary}

\subsection{Non-Trivial Weakly Cauchy Sequences in $JT$}
Here we study the embedding of $J$ into non-reflexive subspaces of $JT$.

We recall (Definition \ref{D2}, Proposition \ref{P1}) that every seminormalized block sequence $(x_n)_n$ which is not weakly null has a subsequence that admits a non-zero $\ell_2$ vector determined by a sequence of reals $(a_i)_{i\in F}$ and a sequence of branches $(\s_i)_{i\in F}$. In the sequel, we will assume that $F = \N$ and the sequence $(|a_i|)_i$ is non-increasing.


\begin{lemma}\label{L7}
Let $(x_n)_n$ be a bounded block sequence in $JT$ admitting a non-zero $\ell_2$ vector determined by the sequences $(a_i)_{i\in \N}$ and $(\s_i)_{i\in \N}$.

a) $x_n \xrightarrow{w^*} x^{**}= \sum_{i=1}^\infty a_i \s_i^{**}$ and $\|x^{**}\| = \big( \sum_{i=1} ^\infty a_i^2\big)^{1/2}$.

b) There exists a further block sequence $(z_n)_n$ consisting of successive convex combinations of $(x_n)_n$ such that $z_n \xrightarrow{w^*} x^{**}$ and $\|z_n\| \to \|x^{**}\|$.
\end{lemma}
\begin{proof}
a) If $x^*$ denotes either some $e_n^*$ or some $\s^*$, it is easy to see that $x^*(x_n) \to x^{**}(x^*)$. Then the first part of (a) follows from Proposition \ref{P3}.
The second part of (a) is explained in subsection 3.3.

b) This is well-known and its proof goes as follows. Assume that $\|x^{**}\| = \alpha$. Then there exists a sequence $(y_n)_n$ in $\alpha B_{JT}$ with $y_n \xrightarrow{w^*} x^{**}$. Observe that $x_n - y_n$ is weakly null; hence there exists a sequence $(z_n)_n$ of successive convex combinations of $(x_n)_n$ and a corresponding sequence $(v_n)_n$ of $(y_n)_n$ such that $\|z_n - v_n\| \to 0$. Since $\|v_n\| \to \alpha$, the result follows.
\end{proof}

\begin{notation}
For $A = \{\s_{i, >m}\}_{i=1}^k$, we say that $A$ is incomparable if the corresponding final segments $\{\s_{i, >m}\}_{i=1}^k$ are incomparable. For $A$ incomparable, we denote by $P_A$ the projection defined as the restriction of any $x\in JT$ to $\cup_{\s_{>m} \in A} \s_{>m}$. As it is shown in Proposition \ref{P_A}, $\|P_A\| = 1$.
\end{notation}

The following is an easy consequence of the definition of the norm of $JT$ and the property that on the branches of $\Tree$ the $JT$-norm coincides with the $J$-norm.

\begin{lemma}\label{L8}
Suppose that $A = \{ \s_{i,>m} : i\in F\}$ is incomparable. Then for each $x\in JT$,
\[ \|P_A(x)\| = \Big( \sum_{\s \in A} \|P_{\s_{> m}}(x)\|_J^2 \Big)^{1/2}, \]
where $\|\cdot\|_J$ denotes the norm of the James space $J$.
\end{lemma}

\begin{remark}\label{R1}
Let $x^{**} = \sum_{i=1}^\infty a_i\s_i^{**}$ with  $\|x^{**}\| = 1$, and let $(x_n)_n$ be a block sequence such that $x_n \xrightarrow{w^*} x^{**}$. In the next proposition, we will assume that $(x_n)_n$ satisfies the following additional properties:
\begin{enumerate}
    \item[i)] $\lim_n \|x_n\| = b \geq 1$.
    \item[ii)] For every $i \in \N$, $\lim_n \|P_{\s_i}(x_n)\| = b_i \geq |a_i|$.
\end{enumerate}
Clearly, every $(x_n)_n$ with $x_n \xrightarrow{w^*} x^{**}$ has a subsequence satisfying (i) and (ii). Moreover, it is easy to check that $(\sum_{i=1}^\infty b_i^2)^{1/2} \leq b$.
\end{remark}

\begin{proposition}\label{P11}
Let $(x_n)_n$ be a block sequence in $JT$ such that
$x_n \xrightarrow{w^*} x^{**}$ with $x^{**} = \sum_{i=1}^\infty a_i \s_i^{**}$ and $\|x^{**}\| = 1$.
Then $(x_n)_n$ has a subsequence equivalent to the basis of $J$ generating a complemented subspace of $JT$.
\end{proposition}
\begin{proof}
As we have noticed, we can assume that $(x_n)_n$ satisfies (i) and (ii) in Remark \ref{R1}. For $0< \epsilon < 1$, we choose a sequence $(\delta_k)_k$ of positive reals such that
\begin{equation}\label{eqn:eq.13}
\delta_1 = b \quad \text{and} \quad \sum_{k=1}^\infty \delta_k < (1 +\epsilon) b.
\end{equation}

Next we partition $\N$ into successive disjoint intervals $(F_k)_k$ satisfying for $k\geq 2$:
\begin{equation}\label{eqn:eq.12}
3(1+\epsilon) \Big(\sum_{i\in F_k}(|a_i| + b_i)^2 \Big)^{1/2} < \delta_k,
\end{equation}
and we choose an increasing sequence of naturals $(m_k)_k$ such that the set
\[ A_k =\Big\{ \s_{i,> m_k} : i \in F_k\Big\} \]
consists of incomparable final segments. Proposition \ref{P_A} yields that $\|P_{A_k}\| = 1$.
By induction, we define a decreasing sequence $(L_k)_k$ of infinite subsets of $\N$ such that the following facts are satisfied. (In the following $(e_n)_n$ denotes the basis of $J$.)

\begin{equation}\label{eqn:eq.11}
\begin{split}
(i) \quad &\forall n\in L_k, \text{ we have } m_k < \supp(x_n) \text{ and } \forall i\in F_k, \quad \Big| \s_i^*(x_n) -a_i \Big| < \epsilon |a_i| \text{ and }  \\
&\quad \Big| \|P_{\s_i}(x_n)\| - b_i \Big| < \epsilon b_i. \\
(ii) \quad &\forall i\in F_k \text{ and scalars }(\lambda_n)_{n\in F_k},\\
&\quad (1-\epsilon) |a_i| \Big\| \sum_{n\in L_k}\lambda_ne_n \Big\|_J\leq \Big\| \sum_{n\in L_k} \lambda_nP_{\s_i}(x_n) \Big\|  \leq  3(1+\epsilon)(|a_i|+ b_i)\Big\| \sum_{n \in L_k} \lambda_ne_n \Big\|. \\
(iii) \quad &\text{For } k\geq 2, \quad \Big\| \sum_{n\in L_k} \lambda_nP_{A_k}(x_n) \Big\| \leq 3(1+\epsilon)\Big(\sum_{i\in F_k}(|a_i|+ b_i)^2\Big)^{1/2}\Big\| \sum_{n \in L_k} \lambda_ne_n \Big\| _J\leq \\ &\quad \leq  \delta_k \Big\| \sum_{n \in L_k} \lambda_ne_n \Big\|_J.
\end{split}
\end{equation}

The inductive definition of $L_k$ satisfying (i) follows from the fact that $(x_n)_n$ is a block sequence, $\|P_{\s_i}(x_n)\| \to b_i$ and $\s_i^*(x_n) \to a_i$.

Property (ii) follows from Proposition \ref{P10} and Property (iii) follows from (ii), Lemma \ref{L8} and the properties of $(\delta_k)_k$ (\ref{eqn:eq.12}).

This completes the inductive definition of the sequence $(L_k)_k$.

Choose an increasing sequence $(n_k)_k$ with $n_k\in L_k$ and set
\[ v_k = P_{\cup_{p\leq k} A_p}(x_{n_k}) \quad \text{and} \quad w_k = x_{n_k} - v_k. \]
We will show the following facts:

1) The block sequence $(v_k)_k$ is dominated by the basis of $J$.

2) The sequence $(w_k)_k$ is either norm null or has a subsequence equivalent to the $\ell_2$ basis.

Notice that if (1) and (2) hold, then there exists a subsequence of $(x_{n_k})_k$ which is equivalent to the basis of $J$. If $(w_k)_k$ is norm null, then this follows from a standard perturbation argument. 
Otherwise, $(w_k)$ is equivalent to the $\ell_2$ basis. Then there exist $c_1, c_2$ positive constants such that
\begin{equation}\label{eqn:eq.10}
\begin{split}
\Big\|\sum_{k=1}^\infty\lambda_k x_{n_k} \Big\| &\leq \Big\| \sum_{k=1}^\infty\lambda_k v_k \Big\| + \Big\| \sum_{k=1}^\infty\lambda_k w_k \Big\| \\
&\leq c_1 \Big\| \sum_{k=1}^\infty\lambda_k e_k \Big\|_J + c_2\Big( \sum_{k=1}^\infty\lambda_k^2 \Big)^{1/2} \leq (c_1+c_2)\Big\| \sum_{k=1}^\infty\lambda_k e_k \Big\|_J.
\end{split}
\end{equation}
The last inequality follows from the fact  that the $\ell_2$ norm is dominated by the $J$ norm in the basis $(e_k)_k$ of $J$. If a subsequence of $(w_k)_k$ is equivalent to the $\ell_2$ basis, then (\ref{eqn:eq.10}) remains valid for this subsequence. Also, if $(w_k)_k$ is norm null, then (1) and a permutation argument yield a subsequence of $(x_{n_k})_k$ dominated by the basis of $J$.

For the opposite inequality, namely the basis of $J$ is dominated by $(x_{n_k})_n$, we notice that $G = \{n_k : k\in\N\} \subset L_1$. Hence (\ref{eqn:eq.11}) (ii) gives:
\[ (1-\epsilon) |a_1| \Big\| \sum_{n\in G}\lambda_ne_n \Big\|_J \leq \Big\| \sum_{n\in G} \lambda_nP_{\s_1}(x_n) \Big\| \leq \Big\| \sum_{n\in G} \lambda_nx_n \Big\|. \]
This ends the proof of the equivalence. It remains to prove (1) and (2).

To prove (1), let $(\lambda_k)_{k\leq m} \subset \R$. By the definition of $(v_k)_k$ we have
\begin{equation}\label{eqn:eq.14}
\sum_{k=1}^m \lambda_k v_k = \sum_{k=1}^m \lambda_k P_{A_1}(v_k) + \sum_{k=2}^m \lambda_k P_{A_2}(v_k) + \dots + \lambda_m P_{A_m}(v_m).
\end{equation}
Since for every $p\geq k$ we have $P_{A_k}(v_p) = P_{A_k}(x_{n_p})$ and $(n_p)_{p\geq k} \subset L_k$ for $1 < k\leq m$, from the bimonotonicity of the basis of $J$ and (\ref{eqn:eq.11}) (iii) we derive:
\begin{equation}\label{eqn:eq.15}
\Big\| \sum_{p=k}^m \lambda_pP_{A_k}(v_p) \Big\| \leq \delta_k \Big\| \sum_{p=k}^m \lambda_p e_p \Big\|_J \leq \delta_k \Big\| \sum_{p=1}^m \lambda_p e_p \Big\|_J.
\end{equation}
For $k=1$ we have
\begin{equation}\label{eqn:eq.17}
\Big\| \sum_{p=1}^m \lambda_pP_{A_1}(v_p) \Big\| \leq 3\Big( \sum_{i\in F_1} (|a_i| + b_i)^2\Big)^{1/2} \Big\| \sum_{p=1}^m \lambda_p e_p \Big\|_J \leq 3\sqrt{2} (1+ \epsilon) b\Big\| \sum_{p=1}^m \lambda_p e_p \Big\|_J.
\end{equation}
Using the triangle inequality, (\ref{eqn:eq.13}) and (\ref{eqn:eq.17}) we get
\begin{equation}\label{eqn:eq.16}
\Big\|\sum_{k=1}^m \lambda_k v_k \Big\| < 6(1+\epsilon)b \Big\|\sum_{k=1}^m \lambda_k e_k \Big\|_J.
\end{equation}
This completes the proof of (1).

To prove (2), we show that $v_k \xrightarrow{w^*} x^{**}$. Indeed, if $\s = \s_i$, then $\s_i^*(v_k) \to a_i = x^{**}(\s_i^*)$. Assume $\s\neq \s_i$ for all $i \in \N$.
We assert that $\lim_k \s^*(v_k) = 0$. Observe from (\ref{eqn:eq.11}) (i) we have that for every $\delta > 0$ there exists a finite initial segment $F$ of $\N$ and $k_0$ such that for $k>k_0$ and $\s \neq \s_i$ for $i\in F$ we have $\| P_\s(v_k)\| < \delta$.

Therefore the assertion is true and $v_k \xrightarrow{w^*} x^{**}$.
Since $x_{n_k} \xrightarrow{w^*} x^{**}$ and $w_k = x_{n_k} - v_k$, we conclude that $w_k \xrightarrow{w} 0$. This and Theorem \ref{A-I} proves (2) and completes the proof that the subsequence $(x_{n_k})_k$ is equivalent to the basis of $J$.

For the complementation of $[(x_{n_k})_k]$ we observe that $P_{\s_1}$ (defined above) is an isomorphism on $[(x_{n_k})_k]$. So, Lemma \ref{compl.J} yields the property.
\end{proof}

\begin{remark}
Although we stated Proposition \ref{P11} for $x^{**} =\sum_{i=1}^\infty a_i \s_i^{**}$, the result remains valid (and is in fact easier to prove) if $x^{**} =\sum_{i\in F}a_i \s_i^{**}$ with $F$ an initial segment of $\N$. Thus, every non-trivial weakly Cauchy block sequence $(x_n)_n$ in $JT$ has a subsequence $(x_{n_k})_k$ equivalent to the basis of $J$ with $[(x_{n_k})_k]$ complemented in $JT$.
\end{remark}

Propositions \ref{P7} and \ref{P11} yield the following result.

\begin{theorem}
Every non-trivial weakly Cauchy sequence in $JT$ has a subsequence equivalent to the basis of $J$
and  the space it generates is complemented in $JT$.
\end{theorem}

\begin{proof}
Let $(z_k)_k$ be a non-trivial weakly Cauchy sequence in $JT$. By Proposition \ref{P7}, it has a subsequence equivalent to a block sequence $(y_k)_k$, which must also be non-trivial weakly Cauchy, and thus $w^*$-converges to some $y^{**} = \sum_{i=1}^\infty a_i\sigma_i^{**}$ with $\|y^{**}\| >0$. By Proposition \ref{P11}, any subsequence of $(y_k)_k$ has a further subsequence equivalent to the basis of $J$ and the space it generates is complemented in $JT$. By the ``moreover'' part of Proposition \ref{P7}, the same conclusion holds for some subsequence of $(z_k)_k$.
\end{proof}

The following results are immediate consequences.

\begin{corollary}\label{TH_2}
Every non-reflexive subspace $Z$ of $JT$ contains isomorphically  a complemented copy of James space $J$.
\end{corollary}

\begin{corollary}
Every Schauder basic sequence in $JT$ has a subsequence which is either equivalent to the basis of $\ell_2$ or to the basis of $J$ and the space it generates complemented in $JT$.
\end{corollary}
We close the paper with the following.
\begin{question}  (i) Is it true that every non reflexive subspace $Y$ of $JT^*$ contains isomorphically the space $J^*$?

(ii) Is it true that every non separable  subspace $Y$ of $JT^*$ contains isomorphically the space $\B$?
\end{question}
Notice that the norm in $J^*$ or $JT^*$ is not explicitly described.
\bigskip

\noindent\textbf{Declarations.}

\begin{itemize}
  \item No funding was received to assist with the preparation of this manuscript.
  \item The authors have no competing interests to declare that are relevant to the content of this article.
\end{itemize}
\bigskip

\noindent
\textbf{Acknowledgement.} We thank the referee for several comments that helped us to improve the paper.

\end{document}